\documentclass[12pt,reqno]{amsart}
\usepackage{color}

\usepackage{amsmath}

\usepackage{amsthm}
\usepackage{amssymb}
\usepackage{amsfonts}
\usepackage{graphics}
\usepackage{latexsym}
\usepackage{comment}
\usepackage{ulem}

\numberwithin{equation}{section}
\newtheorem{thm}{Theorem}[section]
\newtheorem{prop}[thm]{Proposition}
\newtheorem{lem}[thm]{Lemma}
\newtheorem{cor}[thm]{Corollary}
\newtheorem{hyp}{Hypothesis}

\theoremstyle{remark}
\newtheorem{rem}{Remark}[section]
\newtheorem{defn}{Definition}

\newcommand{\BBB}{\mathbb}
\newcommand{\R}{{\BBB R}}
\newcommand{\Z}{{\BBB Z}}
\newcommand{\T}{{\BBB T}}

\newcommand{\N}{{\BBB N}}
\newcommand{\C}{{\BBB C}}
\newcommand{\HT}{{\mathcal H}}%
\newcommand{\LR}[1]{{\langle {#1} \rangle }}

\newcommand{\al}{\alpha}
\newcommand{\be}{\beta}
\newcommand{\ga}{\gamma}

\newcommand{\te}{\theta}

\newcommand{\vp}{\varphi}

\newcommand{\e}{\varepsilon}
\newcommand{\ta}{\tau}

\newcommand{\p}{\partial}

\newcommand{\La}{\Lambda}

\newcommand{\de}{\delta}
\newcommand{\om}{\omega}

\newcommand{\supp}{\operatorname{supp}}
\newcommand{\equivalent}{\Leftrightarrow}

\newcommand{\I}{\infty}

\newcommand{\sgn}{\operatorname{sgn}}

\newcommand{\EQS}[1]{\begin{align} #1 \end{align}}
\newcommand{\EQQS}[1]{\begin{align*} #1 \end{align*}}

\newcommand{\F}{\mathcal{F}}

\newcommand{\ti}{\widetilde}
\newcommand{\ha}{\widehat}

\newcommand{\les}{\lesssim}
\newcommand{\gts}{\gtrsim}

\title[gKdV on $\R$]
{Improved refined bilinear estimates and well-posedness for generalized KdV type equations on $\R$}

\author[L. Molinet and T. Tanaka]{Luc Molinet and Tomoyuki Tanaka}
\address[L. Molinet]{Institut Denis Poisson, Universit\'e de Tours, Universit\'e d'Orl\'eans, CNRS, Parc Grandmont, 37200 Tours, France}
\email[L. Molinet]{luc.molinet@univ-tours.fr}
\address[T. Tanaka]{Graduate School of Engineering Science, Yokohama National University, Yokohama, Kanagawa, 240-8501 Japan}
\email[T. Tanaka]{tanaka-tomoyuki-fp@ynu.ac.jp}

\begin{document}
\setcounter{page}{001}

\begin{abstract}
In this paper, we study the Cauchy problem for one-dimensional dispersive equations posed on $\mathbb{R} $, under the  hypotheses that the dispersive operator behaves, for high frequencies, as a Fourier multiplier by $ i |\xi|^\alpha \xi $ with $ 1 \le \alpha\le 2 $, and that the nonlinear term is of the form $ \partial_x f(u) $ where $f $ is a real analytic function satisfying certain conditions.
We prove the unconditional local well-posedness of the Cauchy problem in $H^s(\mathbb{R}) $  for $ s\ge \frac{5-2\alpha}{4} $  whenever $ 1\le \alpha<\frac{3}{2} $, and for $ s>\frac{1}{2} $ whenever $\alpha\in [\frac{3}{2},2] $.
This result is optimal in the case $\alpha\ge \frac{3}{2}$ in view of the restriction $ s>\frac{1}{2} $ required for the continuous embedding $ H^s(\mathbb{R}) \hookrightarrow L^\infty(\mathbb{R}) $.
The main novelty of this work, compared to our previous studies, is an improvement of the refined linear and bilinear estimates on $\mathbb{R} $.
Our local well-posedness results enable us to derive global existence of solutions for $ \alpha \in [\frac{5}{4},2] $.
\end{abstract}
\maketitle

\section{Introduction}

We continue our study (\cite{MT22, MT3}) of the Cauchy problem associated with the dispersive equation of the form
\EQS{\label{eq1}
  \p_t u + L_{\al+1} u + \p_x(f(u)) &=0,\quad (t,x)\in \R^2,\\
  \label{initial}
  u(0,x)&=u_0(x),\quad x\in\R,
}
under the two following hypotheses on the dispersive linear operator $L_{\alpha+1}$ (with $\alpha\in [1,2]$) and  on the nonlinear function $ f $.

\begin{hyp}\label{hyp1}
$ L_{\alpha+1} $ is a Fourier multiplier operator with symbol  $  - i p_{\al+1} $, where
$p_{\al+1}\in C^{1}(\R)\cap C^2(\R\backslash\{0\})$ is a real-valued odd function satisfying $p_{\al+1}'(\xi)\sim \xi^\al$ and $ p_{\al+1}''(\xi)\sim \xi^{\al-1}$ for all $\xi\ge \xi_0$,  for some $\xi_0>0$.
\end{hyp}

\begin{hyp}\label{hyp2}
  $f:\R\to\R$ is a real analytic function whose Taylor series around the origin has an infinite radius of convergence.
\end{hyp}

It is worth noticing that any polynomial functions satisfy Hypothesis \ref{hyp2}, and this class of equations includes the well-known generalized Korteweg-de Vries (gKdV) and generalized Benjamin-Ono (gBO) equations, which correspond, respectively, to the cases $ L_3=\p_x^3$ and $L_2=-\HT\p_x^2$, and take the form
\EQQS{
  \p_t u + \p_x^3 u  + \p_x (u^{k+1})=0
}
 and
\EQQS{
  \p_t u -  {\mathcal H} \p_x^2 u  + \p_x (u^{k+1})=0,
}
where $ {\mathcal H} $ denotes the Hilbert transform (i.e., the Fourier multiplier by $ -i \sgn(\xi) $).

\subsection{Previous Results on Well-Posedness}

As the equation \eqref{eq1} encompasses many important models, there exists a substantial body of research on its local well-posedness.
In this subsection, we focus on the following equation, which involves a purely dispersive operator and a polynomial nonlinearity:
\EQS{\label{eq4}
  \p_t u -\p_x D_x^\al u +\p_x (u^{k+1})=0,\quad (t,x)\in\R^2,
}
where $1\le k\in \N$, $\al\in [1,2]$ and $\ha{D}_x = |\xi|$.
Recall that $s_c(\al,k)=\frac{1}{2}-\frac{\al}{k}$ is the scaling critical index for \eqref{eq4}.
We mainly summarize previous results in the case $k\ge 3$.
For the cases $k=1,2$, see \cite{LP} and the references therein.

The first attempts to prove the well-posedness of \eqref{eq4} in $ H^s(\R) $ (see, for instance, \cite{ABFS89, BS75, Kato83}) did not make use of the dispersive effects of the equation and were thus restricted to $s>\frac{3}{2} $.
These methods are collectively referred to as the \textit{classical energy method} and play a fundamental role in well-posedness theory.
A first breakthrough was achieved by Kenig, Ponce, and Vega in a series of major works \cite{KPV3,KPV1,KPV2}, where they initiated the study of \eqref{eq4} with the help of local and global smoothing effects of the associated linear group.
In \cite{KPV1}, they showed that the Cauchy problem for \eqref{eq4} is locally well-posed in $H^{\frac{9-3\al}{4}}(\R)$ for $k\ge 1$ by combining these linear smoothing effects with the so-called Kato smoothing effect for nonlinear solutions.
Later, in \cite{KPV2}, they improved this threshold in the case $\al=2$ by using a contraction argument in suitable function spaces adapted to the linear smoothing effects and the degree of the nonlinearity.
More precisely, they refined the result to $H^{\frac{1}{4}}(\R)$ for $k=2$, $H^{\frac{1}{12}}(\R)$ for $k=3$, and $H^{s_c(2,k)}(\R)$ (with small initial data) for $k\ge 4$.
See also \cite{KPV5} for results concerning the modified and generalized Benjamin-Ono equation.

A second breakthrough was accomplished by Bourgain in his seminal work \cite{B93}, where, among other results, he  studied the KdV equation (corresponding to $(\al,k)=(2,1)$) and established the local well-posedness in both $L^2(\T)$ and $ L^2(\R) $, by employing the contraction argument in the space $X^{s,b}$
(see Subsection \ref{subs_FS} for the definition).
Since then, many authors have refined his approach and applied it to a wide range of dispersive equations; see, for instance, \cite{Gru, KPV4, LP}.
In particular, Gr\"unrock \cite{Gru} showed the local well-posedness for \eqref{eq4} with $(\al,k)=(2,3)$ in $H^{s_c(2,3)+}(\R)$.

However, in \cite{MST},  the first author, Saut, and Tzvetkov showed that \eqref{eq4} with $\alpha <2 $ and $k=1$ cannot be solved in the standard Sobolev space $H^s(\R)$ via a contraction argument.
Therefore, for $ \alpha\in [1,2[ $, it seems appropriate to present the results separately for $k=1$ and $ k\ge 2$.
In the case $k=1$, Herr, Ionescu, Kenig, and Koch \cite{HIKK} proved that \eqref{eq4} is locally well-posed in $L^2(\R) $ for any $ \alpha\in ]1,2]$ by combining a para-differential gauge transform with Bourgain's type estimates.
Recall that for $\alpha=1$ (the Benjamin-Ono equation), the local well-posedness in $L^2(\R) $ had already been established in \cite{IK}.
For the case $ k\ge 2$, the first author and Ribaud \cite{MR} established the local well-posedness of \eqref{eq4} with $ \alpha=1$ (the generalized Benjamin-Ono equation) in $H^s(\R) $ for small initial data: for $ s>\frac{1}{2} $ when $k=2$, $s>\frac{1}{3} $ when $k=3 $, and $ s> s_c(1,k)$ when $ k\ge 4$.
For this purpose, they mainly followed the approach developed for the gKdV equation in \cite{KPV2}.
The restriction to small initial data is due to the weakness of the dispersion, with $\alpha=1 $ being a limiting case.
To overcome this weakness, Vento \cite{Vento} combined the contraction method with the gauge transformation introduced by Tao \cite{Tao04}, and established the local well-posedness in $H^{\frac{1}{3}+}(\R)$ when $k= 3$, and in $H^{s_c(1,k)}(\R)$ when $k\ge 4$, for the Cauchy problem associated with \eqref{eq4} in the case $\al=1$.
More recently, Campos, Linares, and Santos \cite{CLS} studied \eqref{eq4} with $\al\in ]1,2]$ and $k\ge 3$.
They  proved the local well-posedness in $H^{s_c(\al,k)+}(\R)$ for $k\ge 3$.
To this end, they followed the approach of \cite{Gru} for $ k=3 $, and that of \cite{KPV2} for $k\ge 4$.
Note that no smallness assumption on the initial data is required, since the dispersion is slightly above the threshold $\al=1$.

Finally, we notice that in the case of more general linear operators and nonlinearities, the authors \cite{MT22,MT3} established the local well-posedness and the unconditional uniqueness of \eqref{eq1} in $H^{s}(\T)$ for $\alpha\in [1,2] $ and $ s\ge 1-\frac{\al}{4}$ with $s>\frac{1}{2}$, under Hypotheses \ref{hyp1}--\ref{hyp2}.
The proof of these results can be carried out in parallel for the real line case.

\subsection{Main Results}
 The main result of this article is the local well-posedness (LWP, for short) for the Cauchy problem \eqref{eq1}--\eqref{initial}, along with the unconditional uniqueness of its solution.
 Before stating our main result precisely, we recall the notion of solutions used in this paper:

\begin{defn}\label{def}
  Let $s>\frac{1}{2}$ and $T>0$.
  We say that $u\in L^\infty(]0,T[;H^s(\R)) $ is a \textit{solution} to \eqref{eq1} associated with the initial datum $ u_0 \in H^s(\R)$  if
  $ u $ satisfies \eqref{eq1}--\eqref{initial} in the distributional sense, i.e., for any test function $ \phi\in C_c^\infty(]-T,T[\times \R) $,  it holds that
  \begin{equation}\label{weakeq}
  \int_0^\infty \int_{\R} \Bigl[(\phi_t +L_{\alpha+1}\phi )u +  \phi_x f(u) \Bigr] \, dx \, dt +\int_{\R} \phi(0,\cdot) u_0 \, dx =0.
  \end{equation}
\end{defn}

\begin{rem}\label{rem_conti}
 Let $u\in L^\infty(]0,T[;H^s(\R)) $ be a solution in the sense of Definition \ref{def}.
 It then holds that $u\in C_w([0,T];H^s(\R))\cap C([0,T];H^{s-}(\R))$ by using the equation \eqref{weakeq}.
 See Remark 2.1 in \cite{MPV19} for details.
\end{rem}

\begin{defn}\label{def_UWP}
 We say that the Cauchy problem associated with \eqref{eq1} is \textit{unconditionally locally well-posed}  in $ H^s(\R )$ if for any initial data $ u_0\in H^s(\R) $, there exist $ T=T(\|u_0\|_{H^s})>0 $ and a solution
 $ u \in C([0,T]; H^s(\R)) $ to \eqref{eq1} emanating from $ u_0 $.
 Moreover, $ u $ is the unique solution to  \eqref{eq1} associated with $ u_0 $ that belongs to $ L^\infty(]0,T[; H^s(\R) )$.
 Finally, for any $ R>0$, the solution-map $ u_0 \mapsto u $ is continuous from the ball in $ H^s(\R) $  with radius $ R $ centered at the origin into $C([0,T(R)]; H^s( \R)) $.
\end{defn}

The uniqueness of solutions in the class $L^\infty(]0,T[; H^s(\R))$ is called \textit{unconditional uniqueness} (UU for short), meaning that the solution is independent of the method of construction.
This notion was first introduced by Kato \cite{Kato95} in the study of the Schr\"odinger equation.
Several UU results are known for KdV-type equations.
Zhou \cite{Zhou} proved UU for the KdV equation in $L^2(\R)$, and Kwon, Oh, and Yoon \cite{KOY20} established UU for the modified KdV equation in $H^{\frac{1}{4}+}(\R)$.
For the equation \eqref{eq1} under Hypotheses \ref{hyp1}--\ref{hyp2}, the authors \cite{MT3} proved UU in $H^{s}(\R)$ for $s\ge 1-\frac{\al}{4}$ with $s>\frac{1}{2}$ and $\al\in [1,2]$.
See also \cite{BIT11, MPV25, MT22, MT24, MP23, P21} for further results on Benjamin-Ono type equations and related models.

We now state our main result.
First we introduce the following index:
 \EQS{\label{def_salpha}
   s(\al)=
   \begin{cases}
     \frac{5-2\al}{4},&1\le \al<\frac{3}{2}\\
     \frac{1}{2},&\frac{3}{2}\le \al \le 2.
   \end{cases}
 }

\begin{thm}\label{theo1}
Let $\alpha \in [1,2]$.
Then, for any $s\ge s(\al)$ with $s>\frac{1}{2}$, the Cauchy problem associated with \eqref{eq1}--\eqref{initial} is unconditionally locally well-posed in $ H^s(\R ) $ with a maximal time of existence $T\ge g(\|u_0\|_{H^{s(\al)\vee(\frac{1}{2}+)}})>0 $, where $ g$ is a smooth decreasing function.
\end{thm}

\begin{rem}
For $\alpha\in ]\frac32,2] $, we can gain a negative power of the highest frequency in all frequency interactions, either by using Bourgain's type estimates, the refined bilinear estimate for separated frequency functions \eqref{eq_bistri2}, or the improved refined bilinear estimate for non-separated frequency functions \eqref{eq_bistri1}.
As a consequence, it is clear that our approach would lead to the unconditional local well-posedness of \eqref{eq1} in $H^{\frac12}(\R) $ when  $ f$ is a polynomial function.
\end{rem}

Next, we discuss global solutions to \eqref{eq1}.
Equation \eqref{eq1} possesses the following conservation laws at the $ L^2$-
and $ H^{\frac{\al}{2}} $-levels:
\EQQS{
M(u)=\int_{\R} u^2 dx \quad \text{and} \quad
E(u)=\frac{1}{2} \int_{\R} u  \partial_x^{-1} L_{\alpha+1} u \, dx
      +\int_{\R} F(u) dx,
}
where $ \partial_x^{-1} L_{\alpha+1} $ is the Fourier multiplier by  $ -\frac{p_{\alpha+1}(\xi)}{\xi}  $,
and
\begin{equation}
  \label{defF} F(x) :=\int_0^x f(y) \, dy \; .
\end{equation}
At this stage, it is worth noticing that Hypothesis \ref{hyp1} ensures that $\xi \mapsto -\frac{p_{\alpha+1}(\xi)}{\xi}$ can be continuously extended to the origin.
Therefore, the restriction of the quadratic part of the energy $ E $ to high frequencies behaves like the $ H^{\frac{\al}{2}}(\R)$-norm, whereas its restriction to low frequencies can be  controlled by the $ L^2$-norm.
Combining these conservation laws with the local well-posedness result above, exactly the same arguments as in [\cite{MT22}, Section 5] lead to the following global well-posedness (GWP) results for \eqref{eq1} when $\al\in[\frac{5}{4},2]$, since $s(\al)\le \frac{\al}{2}\equivalent \al\ge \frac{5}{4}$:

\begin{cor}[GWP for small data] \label{theo2}
 Assume that Hypotheses \ref{hyp1}--\ref{hyp2} hold with $ \alpha \in [\frac{5}{4},2]$.
 Then there exists a constant $ A=A(L_{\alpha+1},f) >0 $ such that, for any initial data $ u_0\in H^s(\R) $ with $ s\ge \frac{\al}{2} $ satisfying $ \|u_0\|_{H^{\frac{\al}{2}}} \le A $, the solution constructed in Theorem \ref{theo1} extends globally in time.
 Moreover, its trajectory is bounded in $ H^{\frac{\al}{2}}(\R) $.
\end{cor}

\begin{cor}[GWP for arbitrarily large data] \label{theo3}
 Assume that Hypotheses \ref{hyp1}--\ref{hyp2} hold with
 $ \alpha \in [\frac{5}{4},2]$.
 Then the solution constructed in Theorem \ref{theo1} extends globally in time if the function $F$ defined
in \eqref{defF}  satisfies one of the following conditions:
 \begin{enumerate}
 \item There exists $C>0  $ such that $|F(x)|\le C (1+|x|^{p+1})$ for some $ 0<p<2\alpha+1  $.
\item   There exists $ B>0  $ such that $ F(x) \le B , \; \forall x\in \R
$.
 \end{enumerate}
 Moreover, its trajectory is bounded in $ H^{\frac{\al}{2}}(\R) $.
\end{cor}

For typical examples of nonlinearities satisfying the assumptions of Corollary \ref{theo3}, see Remark 1.6 in \cite{MT22}.

\subsection{Outline of The Proof}\label{sub13}
Let us outline the main ideas behind the proof of our result.
First, it is worth recalling that, given the generality of our nonlinear term, we cannot expect to get the LWP below $ H^{\frac12+}(\R) $ for \eqref{eq1}.
Moreover, according to \cite{MST}, the solution cannot be obtained by a fixed point argument, and thus it is natural to try to apply an energy method.
When applying the energy method, in \cite{MT3}, the authors divided the contribution of the nonlinear term in the energy estimate into several types of nonlinear interactions, which can be tackled either by Bourgain's type estimates (based on nonresonant relations) or refined bilinear estimates for separated and non-separated frequency functions
(see \eqref{eq_bistri2} for the case of separated frequency functions).
Recall that these refined bilinear estimates are derived not for free solutions of the linear group associated with \eqref{eq1}, but rather for solutions of \eqref{eq1} itself.
It turns out that for any $\alpha\in [1,2] $, Bourgain's type estimates and refined bilinear estimates for \textit{separated} frequency functions enable us to control the corresponding interactions at the $ H^{\frac12+} $-level, even in the periodic setting (cf.\ \cite{MT3}).
Therefore, both on the torus as well as on the real line, to lower the LWP index to $\frac12 + $, we have to improve the treatment of bad interactions for which the refined bilinear estimate for \textit{non-separated} frequency functions is used.
These interactions mainly consist of three high frequency components of order $ N\gg 1 $, two with the same sign and the other with the opposite sign, together with other frequencies that are negligible with respect to $N$.
Note that such interactions may contain resonances.

In this paper, we show how this refined bilinear estimate can be improved in the real line case, thereby enhancing the LWP index in Sobolev spaces.
Let us now explain our improvement.
The idea underlying the refined bilinear estimates for solutions of nonlinear equations goes back to Koch-Tzvetkov \cite{KT1} (see also \cite{BGT1} and \cite{H12} for the same type of approach applied to linear solutions on a compact manifold).
To fix ideas, we start with the well-known bilinear estimates for solutions of the linearized equation associated with \eqref{eq1}.
For $ N_1\ge N_2$, it holds that
 \begin{equation}\label{tata1}
     \|P_{N_1}U_\alpha(t) \varphi_1 P_{N_2}U_\alpha(t) \varphi_2\|_{L^2(I;L^2)}\lesssim N_1^{\frac{1-\alpha}{4}} |I|^{\frac14} \|P_{N_1}\varphi_1\|_{L_x^2}
      \|P_{N_2}\varphi_2\|_{L_x^2},
 \end{equation}
where $ U_\alpha(t) $ is the linear group associated with \eqref{eq1},  $ I$ is a time interval of length $|I|\lesssim 1$, and $ P_N $ is a smooth Fourier projector on spatial frequencies of size $ N $.
Now, a nonlinear solution $u $ to \eqref{eq1} can be written in the Duhamel form as
\begin{equation}\label{tata2}
    P_N u(t) =P_N U_\alpha(t) u_0 + \int_0^t P_N U_\alpha (t-t') \partial_x f(u(t')) \, dt' \; .
\end{equation}
A direct use of \eqref{tata1} to estimate terms like $\|P_{N_1} u P_{N_2}  u \|_{L^2(]0,1[;L^2)} $ in \eqref{tata2} would not be efficient, since the loss of one derivative in the nonlinear term would be too costly.
The idea introduced in \cite{KT1} consists, roughly speaking, in equalizing the contributions of the free and Duhamel parts in \eqref{tata2} by applying the Strichartz estimate on small time intervals of length $ |I| \sim N^{-\gamma} $ with $ \gamma>0$.
Indeed, the loss of one derivative in the nonlinear term can then be compensated by the time integrability of the Duhamel term, which is achieved by gaining a positive power of $ |I|$ via the H\"older inequality in time.
This is particularly efficient when working in an algebra such as $ H^{\frac12+}(\R)$, where the nonlinearity does not cost.
After balancing the two contributions, which corresponds to choosing $\gamma$ so that
\begin{equation}\label{tata5}
N^s \left\|\int_I P_N U_\alpha (-t') \partial_x f(u(t')) \, dt' \right\|_{L^2_x}\lesssim C(\|u\|_{L^\infty_I H_x^s}) \, ,
\end{equation}
we then sum over the time intervals (there are about $ N^\gamma$ such subintervals, contributing roughly $ N^{\frac{\gamma}{2}}$ in total) to get in view of \eqref{tata1},
\begin{equation}\label{tata3}
     \|P_{N_1} u_1 P_{N_2}u_2\|_{L^2(]0,1[;L^2)}
     \lesssim N_1^{\frac{1-\alpha}{4}} N_1^{\frac{\gamma}{4}}  N_1^{-s}N_2^{-s}
     C\left(\sum_{j=1}^2 \|u_j\|_{L^\infty(]0,1[;H^s)}\right),
 \end{equation}
where $ u_1$ and $ u_2 $ are solutions of \eqref{tata2} and $ f$ is assumed to be a polynomial.
Thus, to optimize this estimate, we need to choose $ \gamma\ge 0 $ as small as possible; in other words, we aim to work on the largest possible time intervals.
In \cite{MT3}, the authors took $\gamma=1$, i.e., $ |I|\sim N^{-1} $, in order to get \eqref{tata5}.
In this paper, we prove that on the real line, for $ \alpha\in [1,2]$, we can push this to $\gamma=\max\{2-\alpha, 1-\alpha/3\} $, which results in an improvement for $\alpha\in ]1,2]$.

\begin{rem}
It is worth noticing that for $ \alpha>3/2$, this yields $ \gamma=1-\al/3$, and thus \eqref{tata3} becomes
\begin{align}\label{tata4}
 \begin{aligned}
     \|P_{N_1} u_1 P_{N_2}u_2\|_{L^2(]0,1[;L^2)}
     &\lesssim N_1^{2-\frac43 \alpha}
        N_1^{-s} N_2^{-s}
        C\left(\sum_{j=1}^2
        \|u_j\|_{L^\infty(]0,1[;H^s)}\right) \\
     &\lesssim
         N_1^{-s} N_2^{-s} C\left(\sum_{j=1}^2
         \|u_j\|_{L^\infty(]0,1[;H^s)}\right),
 \end{aligned}
\end{align}
which enables us to treat bad interactions at the $ H^{\frac12+}(\R)$-level.
Indeed, for simplicity, let us take $f(u)=u^k $ with $ k\ge 3$ (bad interactions do appear only for $ k\ge 3$).
After integrating the nonlinear term against $N^{2s} P_N u $, the contribution of the bad interactions in the energy estimate takes the form
\EQQS{
  N^{2s+1} \left|\int_0^t \int_{\R} (P_{\sim N} u)^4 (P_{\ll N} u)^{k-3}
   dxdt'\right|,
}
which can be bounded for $u\in L^\infty_t H^{s}(\R)$ by applying \eqref{tata4} twice, provided that $ 2s>1$, i.e., $ s>1/2$.
\end{rem}

To lower the value of $\gamma$, we have to prove \eqref{tata5} on a time interval $I $ of length greater than $ N^{-1} $.
For simplicity, let us again take $ f(u)=u^k $.
Localizing each function in frequency, we decompose the nonlinear term
$\partial_x P_N u^k $ into terms of the form
\EQQS{
  \partial_x P_N (P_{N_1} u P_{N_2} u \cdots P_{N_k}u),
}
where $ N_1\ge N_2\ge \cdots\ge N_k $.
We divide the analysis into two regions: $ N_2\gtrsim N $ (the non-separated frequency case) and $N_2\ll N $ (the separated frequency case).
It turns out that the limitation $ |I|\sim N^{-1}$ in \cite{MT3} originated from the second region.
Indeed, let us consider a short time interval $I$ with $|I|\sim N^{-\ga}$ for some $\ga>0$.
In the first region $N_2\gtrsim N$, we can use the $ L^4_t L^\infty_x $-Strichartz estimate (see \eqref{stri_2}), together with the H\"older inequality in time, and take advantage of the Sobolev regularity $ s>\frac{1}{2}$ to get
\begin{align}
\label{tata6}
  \begin{aligned}
      &N^s \left\|\int_I  U_\al(-t')
        \partial_x P_N f(u(t')) dt'  \right\|_{L_x^2}  \\
      &\lesssim N^{s+1} N^{\frac{1-\alpha}{4}}
        \|P_{N_1} u P_{N_2} u \cdots
        P_{N_k}u\|_{L^\frac43_I L^1_x} \\
      &\lesssim N^{s+1} N^{\frac{1-\alpha}{4}}
       |I|^{\frac{3}{4}}
       \|P_{N_1} u\|_{L^\infty_T L^2_x}
       \|P_{N_2} u\|_{L^\infty_T L^2_x}
       \prod_{j=3}^k \|P_{N_j}u\|_{L^\infty_{T, x}} \\
      &\lesssim
        \underbrace{N N^{\frac{1-\alpha}{4}} N^{-\frac{3}{4}\gamma} N^{-\frac12}}_{=N^\frac{3-\alpha-3\gamma}{4}}
       \|P_{N_1} u\|_{L^\infty_T H^s_x}
       \|P_{N_2} u\|_{L^\infty_T H^{\frac12}_x}
       \prod_{j=3}^k
       \|P_{N_j}u\|_{L^\infty_{T}H^{\frac12+}_x} \, .
  \end{aligned}
\end{align}
This is acceptable as long as $ \gamma\ge 1-\alpha/3 $.
Now, in the second region $N_2\ll N$ (which implies $N_1\sim N$), if we proceed in a similar way, we cannot obtain a better bound than the following one:
\begin{align}
\label{tata66}
  \begin{aligned}
      N^s \left\|\int_I  U_\al(-t') \partial_x P_N f(u(t')) dt'  \right\|_{L_x^2} & \lesssim N^{s+1}  \|P_{N_1} u P_{N_2} u \cdot\cdot\cdot P_{N_k}u\|_{L^1_I L^2_x}\\
      & \lesssim  N^{1-\gamma}
      \|P_{N_1} u\|_{L^\infty_T H^s_x}
      \prod_{j=2}^k \|P_{N_j}u\|_{L^\infty_{T}H^{\frac12+}_x} \, .
  \end{aligned}
\end{align}
This leads to the restriction $ \gamma\ge 1$.
Therefore, we need to improve this latter estimate.
For this, we slightly modify the above decomposition of $ u^k$ by writing
\begin{align*}
\partial_x P_N&(u^k)
      =k\partial_x P_{\sim N}(P_{N}u (P_{\ll N}u)^{k-1})
       +k\partial_x P_{\sim N}(P_{N}uP_{\gtrsim N }u(P_{\ll N}u)^{k-2})\\
      &\quad+k\partial_x P_{\sim N}([P_N,(P_{\ll N}u)^{k-1}]P_{\sim N}u)
       +\sum_{j=2}^k\binom{k}{j}
       \partial_x P_N((P_{\gtrsim N }u)^j (P_{\ll N}u)^{k-j}).
\end{align*}
The contribution of the first term in the right-hand side gives rise to a term of the form $\|P_N u P_{\ll N} u \|_{L^2_I L^2_x} $, while we can obtain suitable estimates for the other terms.
A key point in improving the bilinear estimate for separated frequency functions is to use this decomposition, namely to estimate $\|P_N u P_{\ll N} u \|_{L^2_I L^2_x} $ recursively.
More precisely, the contributions of the second and fourth terms can be bounded as in the first region of \eqref{tata6}, since there are at least two functions whose frequencies satisfy $ |\xi| \gtrsim N$.
The contribution of the third term can be controlled thanks to a commutator estimate, as follows:
\begin{align*}
N^s  N^{\frac{1-\alpha}{4}} &\|\partial_x P_{\sim N}([P_N,(P_{\ll N}u)^{k-1}]P_{\sim N}u)\|_{L^\frac43_I L^1_x} \\
& \lesssim  N^{\frac{1-\alpha}{4}} N^{-\frac{3}{4}\gamma}
\|P_{\sim N} u\|_{L^\infty_T H^s_x} \|\partial_x (P_{\ll N}u)^{k-1}\|_{L^\infty_T L^2_x} \\
& \lesssim
N^{\frac{1-\alpha}{4}} N^{-\frac{3}{4}\gamma} N^{\frac12}
\|u\|_{L^\infty_T H^s_x}
\prod_{j=2}^k \|P_{N_j}u\|_{L^\infty_{T}H^{\frac12+}_x} \, .
\end{align*}
This yields the same restriction $\gamma\ge 1-\alpha/3$. Finally, the contribution of the first term is estimated by
\EQQS{
  N^s N N^{-\frac{\gamma}{2}}
  \|P_N u P_{\ll N} u \|_{L^2_I L^2_x}
  \prod_{j=2}^k \|P_{N_j}u\|_{L^\infty_{T}H^{\frac12+}_x} \, .
}

Therefore, applying the bilinear estimate for separated frequency functions (see \eqref{eq_bistri2.1}) to the Duhamel formula in order to bound $  \|P_N u P_{\ll N} u \|_{L^2_I L^2_x} $, we eventually obtain, for $ \gamma\ge 1-\alpha/3 $,
\begin{align*}
   &N^s \|P_N u P_{\ll N} u \|_{L^2_I L^2_x} \\
   &\lesssim N^{1-\frac{\al}{2}-\frac{\ga}{2}}
     N^s \|P_N u P_{\ll N} u \|_{L^2_I L^2_x}
    + N^{-\frac{\al}{2}} \|u\|_{L^\infty_T H_x^s}
     C(\|u\|_{L^\infty_T H_x^{\frac12+}}),
\end{align*}
which, for sufficiently large $N$, leads to
\begin{equation}\label{cds}
  N^s \|P_N u P_{\ll N} u \|_{L^2_I L^2_x}
  \lesssim N^{-\frac{\al}{2}} \|u\|_{L^\infty_T H_x^s}
    C(\|u\|_{L^\infty_T H_x^{\frac12+}})
\end{equation}
as soon as $\gamma>2-\alpha $ and $ \gamma\ge 1-\alpha/3 $.
(In fact, the endpoint $\gamma=2-\alpha $ can also be reached by adjusting the length of the time intervals $|I|\sim TN^{-\ga}$, taking $T>0$ sufficiently small.
See also Remark \ref{rem_gamma}.)
We point out that \eqref{cds} is the key estimate to get our improved refined estimates.
Indeed, by reinjecting this bilinear estimate into the second region ($N\sim N_1\gg N_2$), instead of \eqref{tata66}, we obtain, for $\gamma\ge \max\{2-\alpha,1-\alpha/3\}$ and sufficiently large $N$,
\begin{align*}
  N^s \left\|\int_I  U_\al(-t') \partial_x P_N f(u(t')) dt'  \right\|_{L_x^2} & \lesssim N^{s+1}  \|P_{N_1} u P_{N_2} u \cdots P_{N_k}u\|_{L^1_I L^2_x} \\
  & \lesssim N^{s+1-\frac{\ga}{2}}\|P_{N_1} u P_{N_2} u \|_{L^2_{I}L^2_x} \|u\|_{L^\infty_{T,x}}^{k-1} \\
  & \lesssim  N^{1-\frac{\al}{2}-\frac{\ga}{2}}
    \|u\|_{L^\infty_T H_x^s}
    C(\|u\|_{L^\infty_T H_x^{\frac12+}})\\
  & \lesssim   \|u\|_{L^\infty_T H_x^s} C(\|u\|_{L^\infty_T H_x^{\frac12+}}).
\end{align*}
This proves \eqref{tata5} for $\gamma= \max\{2-\alpha,1-\alpha/3\}$, which in turn yields the bilinear estimate \eqref{tata3} for the same range of $\gamma$ (see Proposition \ref{prop_bistri} for details).
Moreover, we observe that \eqref{tata5} also leads to improvements in both the refined linear Strichartz estimate (Proposition \ref{prop_RefStri}) and the refined bilinear estimate for separated frequency functions (Proposition \ref{prop_bistri4}) for this range of $\gamma$.
All these new refined estimates are stated in the beginning of Section \ref{sec_boot}.
Finally, it is worth noticing that we could also have chosen to prove Theorem \ref{theo1} by using the improvement of the refined linear Strichartz estimate stated in Proposition \ref{prop_RefStri}.
Indeed, on the real line, since the linear Strichartz estimate (Proposition \ref{prop_stri1}) associated with $L_{\alpha+1} $ gives the same gain as the bilinear estimate for non-separated frequency functions, our improvement of their associated refined estimates lead to the same estimate on the nonlinear interactions appearing in the energy method.

\subsection*{Plan of This Paper}
This paper is organized as follows.
In Section \ref{notation}, we introduce the notation and collect some fundamental estimates.
In Section \ref{sec_boot}, we present the key estimates of this article.
For this purpose, we first recall some classical results, such as the Strichartz estimate (Proposition \ref{prop_stri1}) and Christ-Kiselev type arguments (Lemma \ref{lem_ck}), and then provide the proofs of the improved  refined estimates, building on the ideas outlined above.
In Section \ref{sec_apri}, we establish the a priori estimate for a solution, while in Section \ref{sec_diff}, we derive the a priori estimate for the difference of two solutions.
The proof of the main result also is sketched.

\section{Notation, Function Spaces and Basic Estimates}\label{notation}

\subsection{Notation}\label{sub_notation}

Throughout this paper, $\N$ denotes the set of non-negative integers.
For any positive numbers $a$ and $b$, we write $a\lesssim b$ for $a\le Cb$ with a constant $C>0$.
We also write $a\sim b$ to mean $a\les b\les a$.
Moreover, we use $a\ll b$ to indicate $Ca\le b$ with a suitably large constant $C>1$.
For two non-negative numbers $a,b$, we denote $a\vee b:=\max\{a,b\}$ and $a\wedge b:=\min\{a,b\}$.
We also write $\LR{\cdot}=(1+|\cdot|^2)^{1/2}$.
For $a\in\R$, $a+$ and $a-$ denote numbers slightly greater and slightly less than $a$, respectively.

For $u=u(t,x)$, $\F u=\tilde{u}$ denotes its space-time Fourier transform, whereas $\F_x u=\hat{u}$ (resp.\ $\F_t u$) denotes its Fourier transform in space (resp.\ time).
We define the Riesz potential by $D_x^s g:=\F_x^{-1}(|\xi|^s \F_x g)$ and the Bessel potential by $J_x^s g:=\F_x^{-1}(\LR{\xi}^s \F_x g)$.
We also denote the unitary group associated to the linear part of \eqref{eq1} by $U_\al(t)=e^{-tL_{\al+1}}$, i.e.,
\EQQS{
  U_\al(t)u=\F_x^{-1}(e^{itp_{\al+1}(\xi)}\F_x u).
}
Throughout this paper, we fix a smooth even cutoff function $\chi$:
let $\chi\in C_0^\I(\R)$ satisfy
\EQS{\label{defchi}
  0\le \chi\le 1, \quad \chi|_{-1,1}=1\quad
  \textrm{and}\quad \supp\chi\subset[-2,2].
}
We set $\phi(\xi):=\chi(\xi)-\chi(2\xi)$.
For any $l\in\N$, we define
\EQS{\label{defpsi}
  \phi_{2^l}(\xi):=\phi(2^{-l}\xi),\quad
  \psi_{2^l}(\ta,\xi):=\phi_{2^l}(\ta-p_{\al+1}(\xi)),
}
where $ip_{\al+1}(\xi)$ is the Fourier symbol of $L_{\al+1}$.
By convention, we also denote
\EQQS{
  \phi_0(\xi)=\chi(2\xi)\quad
  \textrm{and}\quad
  \psi_0(\ta,\xi)=\chi(2(\ta-p_{\al+1}(\xi))).
}
Any summations over capitalized variables such as $K, L, M$ or $N$ are presumed to be dyadic.
We work with non-homogeneous dyadic decompositions, i.e., these variables
ranges over numbers of the form $\{2^k; k\in\N\}\cup \{0\}$.
We call those numbers \textit{non-homogeneous dyadic numbers}.
It is worth pointing out that $\sum_N\phi_N(\xi)=1$ for any $\xi\in\R$,
\EQQS{
  \supp(\phi_N)\subset\{N/2\le |\xi|\le 2N\},\ N\ge 1,\quad
  \textrm{and}\quad
   \supp(\phi_0)\subset\{|\xi|\le 1\}.
}

Finally, we define the Littlewood--Paley multipliers $P_N$ and $Q_L$ by
\EQQS{
  P_N u=\F_x^{-1}(\phi_N \F_x u) \quad\textrm{and}
  \quad Q_Lu=\F^{-1}(\psi_L \F u).
}
We also set
$P_{\ge N}:=\sum_{K\ge N}P_K$,
$P_{\le N}:=\sum_{K\le N}P_K, Q_{\ge L}:=\sum_{K\ge L}Q_K$ and $Q_{\le L}:=\sum_{K\le L}Q_K$.

\subsection{Function Spaces}\label{subs_FS}
For $1\le p\le \I$, $L^p(\R)$ is the standard Lebesgue space with the norm $\|\cdot\|_{L^p}$.

In this paper, we will use the frequency envelope method (see for instance \cite{KT1, Tao04}) in order to show the continuity result with respect to initial data.
To this aim, we first introduce the following:
\begin{defn}\label{def_AFW}
  Let $\de>1$.
  A dyadic positive sequence $\{\om_N^{(\de)}\}_{N\in 2^\N\cup\{0\}}$ is called an \textit{acceptable frequency weight} if it satisfies $\om_0^{(\de)}=1$ and $\om_N^{(\de)}\le \om_{2N}^{(\de)}\le \de\om_N^{(\de)}$ for all $N\ge 1$, with $\de\le 2$.
  We simply write $\{\om_N\}$ when no confusion arises.
\end{defn}
With an acceptable frequency weight $\{\om_N\}$, we slightly modulate the classical Sobolev spaces in the following way:
for $s\ge0$, we define $H_\om^s(\R)$ with the norm
\EQQS{
  \|u\|_{H_\om^s}
  :=\bigg(\sum_{N\in 2^{\N}\cup\{0\}}\om_N^2 (1\vee N)^{2s}\|P_N u\|_{L^2}^2\bigg)^{\frac 12}.
}
Note that $H_\om^s(\R)=H^s(\R)$ when we choose $\om_N\equiv 1$.
Here, $H^s(\R)$ is the usual $L^2$--based Sobolev space.
If $B_x$ is one of spaces defined above, for $1\le p\le \I$ and $T>0$, we define the space--time spaces $L_t^p B_x :=L^p(\R;B_x)$ and $L_T^p B_x :=L^p(]0,T[;B_x)$ equipped with the norms (with obvious modifications for $p=\I$)
\EQQS{
  \|u\|_{L_t^p B_x}=\bigg(\int_\R\|u(t,\cdot)\|_{B_x}^p dt\bigg)^{\frac 1p}\quad
  \textrm{and}\quad
  \|u\|_{L_T^p B_x}=\bigg(\int_0^T\|u(t,\cdot)\|_{B_x}^p dt\bigg)^{\frac 1p},
}
respectively.
For $s,b\in\R$, we introduce the Bourgain spaces $X^{s,b}$ associated to the operator $L_{\al+1}$ endowed with the norm
\EQQS{
  \|u\|_{X^{s,b}}
  =\left(\iint_{\R^2} \LR{\xi}^{2s}\LR{\ta-p_{\al+1}(\xi)}^{2b}|\tilde{u}(\ta,\xi)|^2d\ta d\xi\right)^{\frac 12}.
}
We also use a slightly stronger space $X_\om^{s,b}$ with the norm
\EQQS{
  \|u\|_{X_\om^{s,b}}
  :=\bigg(\sum_{N\in 2^\N\cup\{0\}}\om_N^2 (1\vee N)^{2s}\|P_N u\|_{X^{0,b}}^2\bigg)^{\frac 12}.
}
In the proof of the improved bilinear Strichartz estimate, we use the Besov type $X^{s,b,q}$ spaces: for $b\in\R$ and $1\le q<\I$,
\EQQS{
  \|u\|_{X^{0,b,q}}:=\bigg(\sum_{L\in 2^\N \cup\{0\}}(1\vee L)^{bq}\|Q_L u\|_{L_{t,x}^2}^q\bigg)^{\frac 1q}
}
with the obvious modifications in the case $q=\I$.
However, we only use $X^{0,\frac12,1}$ throughout this paper.
We define the function spaces $Z^s $ (resp.\ $Z^s_\om $), with $s\in \R$, as $Z^s:= L_t^\I H^s\cap X^{s-1,1}$ (resp.\ $Z^s_\om:= L_t^\I H_\om^s\cap X_\om^{s-1,1}$), endowed with the natural norm
\EQQS{
  \|u\|_{Z^s}=\|u\|_{L_t^\I H^s}+\|u\|_{X^{s-1,1}} \quad
  (\text{resp}.\  \|u\|_{Z^s_\om}=\|u\|_{L_t^\I H^s_\om}+\|u\|_{X^{s-1,1}_\om}) .
}
We also use the restriction in time versions of these spaces.
Let $T>0$ be a positive time and $B$ be a normed space of space-time functions.
The restriction space $B_T$ will be the space of functions $u:]0,T[\times\R\to\R$ or $\C$ satisfying
\EQQS{
  \|u\|_{B_T}
  :=\inf\{\|\tilde{u}\|_B \ |\ \tilde{u}:\R^2\to\R\ \textrm{or}\ \C,\ \tilde{u}=u\ \textrm{on}\ ]0,T[\times\R\}<\I.
}

Finally, we introduce a bounded linear operator from $X_{\om,T}^{s-1,1}\cap L_T^\I H_\om^s$ into $Z_\om^s$ with a bound which does not depend on $s$ and $T$. The existence of this operator ensures that actually $ Z^s_{\om,T}= L_T^\I H^s_\om\cap X^{s-1,1}_{\om,T}$.
Following \cite{MN08}, we define $\rho_T$ as
\EQS{\label{def_ext}
  \rho_T(u)(t):=U_\al(t)\chi(t)U_\al(-\mu_T(t))u(\mu_T(t)),
}
where $\mu_T$ is the continuous piecewise affine function defined by
\EQS{
  \mu_T(t)=
  \begin{cases}
    0 &\textrm{for}\quad t\notin]0,2T[,\\
    t &\textrm{for}\quad t\in [0,T],\\
    2T-t &\textrm{for}\quad t\in [T,2T].
  \end{cases}
}

\begin{lem}\label{extensionlem}
  Let $0<T\le 1$, $s\in\R$ and let $\{\om_N\}$ be an acceptable frequency weight.
  Then,
  \EQQS{
    \rho_T:&X_{\om,T}^{s-1,1} \cap L_T^\I H_\om^s\to Z^s_\om\\
    &u\mapsto \rho_T(u)
  }
  is a bounded linear operator, i.e.,
  \EQS{\label{eq2.1}
    \|\rho_T(u)\|_{L_t^\I H_\om^s}
    +\|\rho_T(u)\|_{X^{s-1,1}_\om}\lesssim
    \|u\|_{L_T^\I H_\om^s}
    +\|u\|_{X_{\om,T}^{s-1,1}},
  }
  for all $u\in X_{\om,T}^{s-1,1}\cap L_T^\I H_\om^s$.
  Moreover, it holds that
  \EQS{\label{eq2.1single}
  \|\rho_T(u)\|_{L_t^\I H_\om^s}
  \lesssim
  \|u\|_{L_T^\I H_\om^s}
  }
  for all $u\in L_T^\I H_\om^s$.
  Here, the implicit constants in \eqref{eq2.1} and \eqref{eq2.1single} can be chosen independent of $0<T\le 1$ and $s\in\R$.
\end{lem}

\begin{proof}
  See Lemma 2.4  in \cite{MPV19} for $ \om_N\equiv 1$ but it is obvious that the result does not depend on $ \om_N$.
\end{proof}

\subsection{Basic Estimates}

In this subsection, we collect some fundamental estimates.
Well-known estimates are adapted for our setting $H_\om^s(\R)$ and $f(u)$.

\begin{lem}
  Let $\{\om_N\}$ be an acceptable frequency weight.
  Then we have the estimate
  \EQS{\label{eq2.2}
    \|uv\|_{H_\om^s}
    \lesssim \|u\|_{H_\om^s}\|v\|_{L^\I}+\|u\|_{L^\I}\|v\|_{H_\om^s},
  }
  whenever $s>0 $ or $ s\ge 0 $ and $ \omega_N \equiv 1$.
  In particular, for any fixed real smooth function $ f $ with $ f(0)=0 $, there exists a real smooth function $ G=G[f] $ that is increasing and non-negative on $ \R_+ $  such that
    \EQS{\label{eq2.2ana}
    \|f(u)\|_{H_\om^s}
    \lesssim   G(\|u\|_{L^\I}) \|u\|_{H_\om^s},
  }
   whenever $s>0 $ or $ s\ge 0 $ and $ \omega_N \equiv 1$.
 \end{lem}

\begin{proof}
  See [\cite{MT22}, Lemma 2.2].
\end{proof}

\begin{lem}
  Assume that $s_1+s_2\ge 0, s_1\wedge s_2\ge s_3, s_3<s_1+s_2-1/2$.
  Then
  \EQS{\label{eq2.3}
    \|uv\|_{H^{s_3}}\lesssim \|u\|_{H^{s_1}}\|v\|_{H^{s_2}}.
  }
  In particular, for $u,v\in H^s(\R)$ with $s>1/2 $ and any fixed real smooth function $f$, there exists a real smooth function $ G=G[f] $ that is increasing and non-negative on $ \R_+ $ such that
   \EQS{\label{eq2.3ana}
    \|f(u)-f(v)\|_{H^\theta}\le G(\|u\|_{H^{s}}+ \|v\|_{H^s})\|u-v\|_{H^{\theta}}.
  }
  for $ \theta\in \{0, s-1\}$.
\end{lem}
\begin{proof}
  For \eqref{eq2.3}, see [\cite{GLM14}, Lemma 3.4].
  The proof of \eqref{eq2.3ana} can be found in [\cite{MT22}, Lemma 2.3].
\end{proof}

We also use the following Leibniz rule.
In our setting, it is important to use this inequality with $r=1$.

\begin{lem}
    Let $1\le r<\I$, $1<p_1,p_2,q_1,q_2\le \I$ satisfy $\frac{1}{r}= \frac{1}{p_1}+\frac{1}{q_1}=\frac{1}{p_2}+\frac{1}{q_2}$.
    Given $s>0$, there exists $C=C(n,s,r,p_1,p_2,q_1,q_2)<\I$ such that for all $u,v\in \mathcal{S}(\R)$ we have
    \EQS{\label{leibniz1}
      \|J_x^{s}(uv)\|_{L^r}
      \le C\|J_x^{s}u\|_{L^{p_1}}\|v\|_{L^{q_1}}
        +C\|u\|_{L^{p_2}}\|J_x^{s} v\|_{L^{q_2}}.
    }
    Moreover, for any fixed real smooth function $f$ with $f(0)=f'(0)=0$, there exists a real smooth function $ G=G[f] $ that is increasing and non-negative on $ \R_+ $ such that
    \EQS{\label{leibniz2}
      \|J_x^s f(u)\|_{L^1}
      \le G(\|u\|_{H^{\frac{1}{2}+}})\|u\|_{H^s}.
    }
\end{lem}

\begin{proof}
    See [\cite{GO14}, Theorem 1] for the proof of \eqref{leibniz1}.
    By \eqref{eq2.2} and \eqref{leibniz1}, we have
    \EQQS{
      \|J_x^s (u^k)\|_{L^1}
      \le C\|J_x^s u\|_{L^2}\|u^k\|_{L^2}
        +C\|u\|_{L^2}\|J_x^s (u^{k-1})\|_{L^2}
      \le C^k \|u\|_{H^{\frac{1}{2}+}}^{k-1}
        \|u\|_{H^s}.
    }
    The estimate \eqref{leibniz2} follows from a Taylor expansion and the triangle inequality.
\end{proof}

We will  frequently use the following lemma, which can be seen as a variant of integration by parts.

\begin{lem}\label{lem_comm1}
  Let $N\in 2^{\N}\cup\{0\}$.
  Then,
  \EQQS{
    \left|\int_\R \Pi(u,v)wdx\right|
    \lesssim\|u\|_{L_x^2}\|v\|_{L_x^2}\|\p_x w\|_{L_x^\I},
  }
  where the implicit constant is independent of $u, v, w$ and $N$, and
  \EQS{\label{def_pi}
    \Pi(u,v):=v \p_x P_N^2u +u \p_x P_N^2v .
  }
\end{lem}

\begin{proof}
  See [\cite{MT22}, Lemma 2.4].
\end{proof}
We also frequently use the following commutator estimate.
One can derive \eqref{eq_comm2} by using a Coifman-Meyer type theorem, in particular [\cite{MPTT}, Theorem 1.1].
However, we present a simpler proof, following the argument of Kishimoto \cite{K25}.

\begin{lem}\label{lem_comm2}
  Let $M, N\in 2^{\N}\cup\{0\}$.
  Then,
  \EQS{\label{eq_comm2}
    \|[P_{N}, P_{M} v]P_{\sim N}u\|_{L^1}
    \les \frac{1\vee M}{1\vee N}
     \| P_{M} v\|_{L^2}\|P_{\sim N} u\|_{L^2},
  }
  where the implicit constant is independent of $u, v, M$ and $N$.
\end{lem}

\begin{proof}
  We follow the argument in [\cite{K25}, Lemma 2.9].
  It suffices to consider the case $N\gg M$ and $N\gg 1$.
  We first assume that $M\gg 1$.
  Let $\sigma_{M,N}\in C^\I(\R^2)$ be a function defined by
  \EQS{\label{def_sigma}
    \sigma_{M,N}(\xi_1,\xi_2)
  =\tilde{\phi}_{M}(\xi_1)\tilde{\phi}_{\sim N}(\xi_2)
   (\phi_{N}(\xi_1+\xi_2)-\phi_{N}(\xi_2))
   \cdot \frac{N}{M},
  }
  where $\tilde{\phi}_{M}(\xi_1)=\phi_{M/2}(\xi_1)+\phi_{M}(\xi_1)+\phi_{2M}(\xi_1)$ and $\tilde{\phi}_{\sim N}(\xi_2)=\sum_{j=1}^5 \phi_{2^{j-3}N}(\xi_2)$.
  Then, by elementary calculus, there exists $C>0$ independent of $M$ and $N$ such that
  \EQS{\label{cond_sigma}
    \forall \be=(\be_1,\be_2)\in\N^2,\quad
    |\p_{\xi_1}^{\be_1}\p_{\xi_2}^{\be_2}
     \sigma_{M,N}(\xi_1,\xi_2)|
    \le CM^{-\be_1}N^{-\be_2},
  }
  We also define a $C^\I$ function $\psi_{M,N}$ on $\R^2$ by $\psi_{M,N}(\eta_1,\eta_2)
    :=\sigma_{M,N}(M\eta_1,N\eta_2)$.
  We write $\psi_{M,N}=\psi$ if there is no risk of confusion.
  By the inverse Fourier transform, we obtain
  \EQQS{
    \sigma_{M,N}(\xi_1,\xi_2)
      =\psi_{M,N}\Big(\frac{\xi_1}{M},\frac{\xi_2}{N}\Big)
      =\int_{\R^2}\hat{\psi}_{M,N}(x_1,x_2)
        e^{i(\frac{\xi_1}{M}x_1+\frac{\xi_2}{N}x_2)}dx_1dx_2,
  }
  where $\hat{\psi}_{M,N}$ is the Fourier transform of $\psi_{M,N}$.
  Then, we have
  \EQQS{
    &\frac{N}{M}\|[P_{N}, P_{M} v]P_{\sim N}u\|_{L^1}\\
    &=\bigg\|\F^{-1}\bigg[ \int_{\xi=\xi_1+\xi_2}
      \sigma_{M,N}(\xi_1,\xi_2)
      \F[P_{M} v](\xi_1)\F[P_{\sim N} u](\xi_2) d\xi_{1,2}
      \bigg]\bigg\|_{L^1}\\
    &\le \int_{\R^2}|\hat{\psi}(x_1,x_2)|
      \bigg\|\F^{-1}\bigg[ \int_{\xi=\xi_1+\xi_2}
        e^{i\frac{x_1}{M}\xi_1}
        \F[P_{M} v](\xi_1)
        e^{i\frac{x_2}{N}\xi_2}
        \F[P_{\sim N} u](\xi_2)
        \bigg]\bigg\|_{L^1}dx_1dx_2\\
    &=C\int_{\R^2}|\hat{\psi}(x_1,x_2)|
      \Big\|(P_{M} v)\Big(\cdot+\frac{x_1}{M}\Big)
        (P_{\sim N} u)\Big(\cdot+\frac{x_2}{N}\Big)\Big\|_{L^1}
        dx_1dx_2\\
    &\le C\|\hat{\psi}\|_{L^1(\R^2)}\|P_{M} v\|_{L^2}
      \|P_{\sim N} u\|_{L^2}.
  }
  Therefore, it suffices to show that $\|\hat{\psi}\|_{L^1(\R^2)}\le C$ for some $C>0$, where the constant $C$ is independent of $M,N$.
  By the Sobolev inequality, we have
  \EQQS{
    \|\hat{\psi}\|_{L^1(\R^2)}
    \les \|\psi\|_{H^{2}(\R^2)}
    \les \|\psi\|_{L^2(\R^2)}+\|\p_{\xi_1}^2\psi\|_{L^2(\R^2)}
      +\|\p_{\xi_2}^2\psi\|_{L^2(\R^2)}.
  }
  All these terms are bounded uniformly, thanks to the definition of $\psi$ and the condition \eqref{cond_sigma}.
  We note that the above computation also holds in the case $M\les 1$, with a slight modification.
  This completes the proof.
\end{proof}

\section{Improved refined bilinear and linear Strichartz estimates}
\label{sec_boot}

In this section, we prove the main estimates that play a crucial role in the proof of the a priori estimate (Proposition \ref{prop_apri1}).
To this end, we first recall the standard Strichartz estimate associated with our operator $L_{\al+1}$.
The following estimate is essentially proved in Theorem 2.1 of \cite{KPV3}.

\begin{prop}\label{prop_stri1}
  Let $0<T<2$ and $\al\in[1,2]$, and $\theta_j=0,1$ for $j=1,2$.
  Then, it holds that for any $\vp\in L_x^2(\R)$,
  \EQS{\label{stri_1}
    \left\|D_x^{\frac{\al-1}{4}}U_\al(t)\vp\right\|_{L_T^4 L_x^\I}
    \lesssim \|\vp\|_{L_x^2}.
  }
  Moreover, it holds that for any $f\in L^{\frac{4}{4-\te_2}}(]0,T[; L^{\frac{2}{1+\te_2}})$
  \EQS{\label{stri_2}
    \bigg\|\int_0^T D_x^{\frac{\al-1}{4}(\te_1+\te_2)} U_\al(t-t') f(t',x) dt'\bigg\|_{L_T^{\frac{4}{\te_1}} L_x^{\frac{2}{1-\te_1}}}
    \lesssim \|f\|_{L_T^{\frac{4}{4-\te_2}} L_x^{\frac{2}{1+\te_2}}}.
  }
\end{prop}

To take into account certain special terms, such as commutator terms, we introduce the following notation:

\begin{defn}
  For $u,v\in L^2(\R)$ and $a\in L^\I(\R^2)$, we set
  \EQS{\label{def_lambda}
    \F_x(\La_a(u,v))(\xi)
    :=\int_{\xi_1+\xi_2=\xi} a(\xi_1,\xi_2)
      \hat{u}(\xi_1)\hat{v}(\xi_2) d\xi_1.
  }
  Remark that when $a\equiv1$, we have $\La_a(u,v)=uv$.
\end{defn}
As indicated in the Introduction, in the case of separated frequency interactions, we will make use of the following  quite classical refined bilinear Strichartz estimate.
Its proof essentially follows that of Proposition 3.2
in \cite{MT3}, using \eqref{cor2} of Corollary \ref{cor_free1} stated below.
\begin{prop}[Refined bilinear Strichartz]\label{prop_bistri2}
  Let $0<T<1$ and $\al\in[1,2]$ and $N_1\vee N_2\gg \LR{N_1\wedge N_2}$.
  Let $f_1,f_2\in L^\I(]0,T[;L^2(\R))$ and $a\in L^\I(\R^2)$ be such that $\|a\|_{L^\I}\lesssim 1$.
  Let $u_1,u_2\in C([0,T];L^2(\R))$ satisfying
  \EQQS{
    \p_t u_j + L_{\al+1} u_j +\p_x f_j=0
  }
  on $]0,T[\times \R$ for $j=1,2$.
  Then for any $ \theta\in[0,1] $ it holds
  \EQS{\label{eq_bistri2}
    \begin{aligned}
      \|\La_a(P_{N_1} u_1,&
        P_{N_2} u_2)\|_{L_{T,x}^2}  \lesssim T^{\frac{\theta-1}{2}}
        (N_1\vee N_2)^{-\frac{(1-\theta)(\al-1)}{2}} \LR{N_1\wedge N_2}^{\frac \theta2} \\
     &\times(\|P_{N_1}u_1\|_{L^2_{T,x}}+\|P_{N_1}f_1\|_{L^2_{T,x}})(\|P_{N_2} u_2\|_{L_T^\infty L^2_x}+\|P_{N_2} f_2\|_{L_T^\infty L^2_x}).
    \end{aligned}
  }
\end{prop}

\begin{rem}
  In \cite{MT3}, in the periodic setting one of the main estimates is the following:
  \EQS{\label{eq_bistri2.1}
    \begin{aligned}
      \|\La_a(P_{N_1} u_1,&
        P_{N_2} u_2)\|_{L_{T,x}^2}  \lesssim T^{\frac{\theta-1}{2}}
        \LR{N_1\wedge N_2}^{\frac \theta2} \\
     &\times(\|P_{N_1}u_1\|_{L^2_{T,x}}+\|P_{N_1}f_1\|_{L^2_{T,x}})(\|P_{N_2} u_2\|_{L_T^\infty L^2_x}+\|P_{N_2} f_2\|_{L_T^\infty L^2_x})
    \end{aligned}
  }
  under the same assumption as in Proposition \ref{prop_bistri2}.
  It is worth noticing that on $\R $, in sharp contrast to the periodic setting, we obtain a regularizing effect as soon as $\alpha>1 $.
  However, \eqref{eq_bistri2.1} is already strong enough to deal with some nonlinear interactions at the regularity $s>1/2$.
  Note that \eqref{eq_bistri2} essentially follows from the proof of \eqref{eq_bistri2.1}.
  The only difference is that we use \eqref{bil2} instead of (3.6) in \cite{MT3}, which is the corresponding estimate on $\T$ to \eqref{bil2}.
\end{rem}

Let us now state our improved refined estimates.
As explained in Subsection \ref{sub13}, the cornerstone of our improvement lies in the observation that the bilinear estimate on small intervals of length $ N^{-1} $ for solutions to \eqref{eq1}, projected onto separated frequencies, may be extended to larger intervals, namely of length $N^{-\ga}$ with $\ga=\max\{2-\al,1-\al/3\}$.
This is proven further below in Proposition \ref{prop_bistri3}.
Note that this extension seems not to be possible on $\T$.
As indicated in the Introduction, we could equivalently prove our UU result by using the improved refined bilinear estimate for non-separated frequency functions or the improved refined linear (Strichartz) estimate.
We point out that all these improved refined estimates are consequences of Proposition \ref{prop_bistri3}.

We start with the improved refined bilinear estimates for non-separated frequency functions, which will be used in the energy estimates.

\begin{prop}[Improved refined bilinear estimate I]\label{prop_bistri}
  Let $0<T<1$, $\al\in[1,2]$, $\ga=\max\{2-\al,1-\al/3\}$, $N_1\ge N_2\ge 1$, $s>1/2$, and $0<\e\ll 1$ so that $s>1/2+\e$.
  Let $a\in L^\I(\R^2)$ be such that $\|a\|_{L^\I}\les 1$.
  Finally, let $u_1,u_2\in L^\I(]0,T[;H^s(\R))$ satisfying \eqref{eq1} on $]0,T[\times \R$, with $f $ satisfying Hypothesis \ref{hyp2}, and for some $K\ge 1$,
  \EQQS{
    \|u_1\|_{L_T^\I H_x^{\frac{1}{2}+\e}}
    +\|u_2\|_{L_T^\I H_x^{\frac{1}{2}+\e}}\le K.
  }
  Then for sufficiently small $T>0$ satisfying $G(K)T^{\frac{1}{2}}\ll 1$, it holds that
  \EQS{\label{eq_bistri1}
    \begin{aligned}
      &\| \La_a(P_{N_1} u_1, P_{N_2} u_2)\|_{L_{T,x}^2}\\
      &\le G(K) T^{\frac{1}{4}}
       N_1^{\frac{\ga+1-\al}{4}}N_2^{-s}
       \Big(T^{-\frac{1}{2}}\|P_{N_1}u_1\|_{L_{T,x}^2}
       +N_1^{-s-\e}\|u_1\|_{L_T^\I H_x^s}\Big)
       \|u_2\|_{L_T^\I H_x^s},
    \end{aligned}
  }
  where $G:\R_{\ge 0}\to \R_{\ge 0}$ is a smooth increasing function depending only on $ f$ and on the Fourier symbol $-i p_{\alpha+1} $ of $ L_{\alpha+1}$.
\end{prop}

\begin{rem}\label{rem_gamma}
  We make some remarks regarding the choice of $\ga$.
  \begin{enumerate}
    \item Note that $\gamma=\max\{2-\al,1-\al/3\}$ leads to $\gamma=1-\alpha/3 $ for $\gamma\in [1,3/2]$ and
    $ \gamma=2-\alpha $ for $ \alpha\in [1,3/2]$.
    \item For $\al>3/2$, it is also possible to obtain \eqref{eq_bistri1} by taking $N_1$ sufficiently large, rather than making $T$ sufficiently small.
    This is because, in the proof of \eqref{eq_sep1}, we need to ensure the quantity $N_1^{\frac{2-\al-\ga}{2}}T$ to be sufficiently small.
    For simplicity, we treat both cases in a unified manner.
    \item It is worth noting that, again in the case $ \alpha>3/2$,  \eqref{eq_bistri1} enables us to gain a negative power of the high frequency, namely $N_1^{-(\frac{\al}{3}-\frac{1}{2})}$.
  \end{enumerate}
\end{rem}

The  corresponding improved refined linear (Strichartz) estimate can be stated as follows:

\begin{prop}[Improved refined linear estimate]\label{prop_RefStri}
  Let $0<T<1$, $\al\in[1,2]$, $\ga=\max\{2-\al,1-\al/3\}$, $N\ge 1$, $s>1/2$, and $0<\e\ll 1$ so that $s>1/2+\e$.
  Let $u\in L^\I(]0,T[;H^s(\R))$ satisfying \eqref{eq1} on $]0,T[\times \R$, with $f $ satisfying Hypothesis \ref{hyp2}, and $\|u\|_{L_T^\I H_x^{\frac{1}{2}+\e}}\le K$ for some $K\ge 1$.
  Then for sufficiently small $T>0$ satisfying $G(K)T^{\frac{1}{2}}\ll 1$, it holds that
  \EQS{\label{eq_RefStri1}
    \| P_{N} u\|_{L_T^2 L_x^\I}
    \le G(K) N^{\frac{\ga+1-\al}{4}}
       \left(T^{-\frac{1}{4}}\|P_{N}u\|_{L_{T,x}^2}
       +T^{\frac{3}{4}}N^{-s-\e}\|u\|_{L_T^\I H_x^s}\right),
  }
  where $G:\R_{\ge 0}\to \R_{\ge 0}$ is a smooth increasing function depending only on $ f$ and on the Fourier symbol $-i p_{\alpha+1} $ of $ L_{\alpha+1}$.
\end{prop}

Finally, we state the improved refined bilinear estimate for separated frequency functions. However, as explained above Proposition \ref{prop_bistri2} is sufficient for our purpose so that we will not make use of this proposition in this paper.

\begin{prop}[Improved refined bilinear estimate II]\label{prop_bistri4}
  Let $0<T<1$, $\al\in[1,2]$, $\ga=\max\{2-\al,1-\al/3\}$, $N_1\gg N_2\ge 1$, $s>1/2$, and $0<\e\ll 1$ so that $s>1/2+\e$.
  Let $a\in L^\I(\R^2)$ be such that $\|a\|_{L^\I}\les 1$.
  Finally, let $u_1,u_2\in L^\I(]0,T[;H^s(\R))$ satisfying \eqref{eq1} on $]0,T[\times \R$, with $f $ satisfying Hypothesis \ref{hyp2}, and for some $K\ge 1$,
  \EQQS{
    \|u_1\|_{L_T^\I H_x^{\frac{1}{2}+\e}}
    +\|u_2\|_{L_T^\I H_x^{\frac{1}{2}+\e}}\le K.
  }
  Then, for sufficiently small $T>0$ satisfying $G(K)T^{\frac{1}{2}}\ll 1$, it holds that
  \EQQS{
    &\| \La_a(P_{N_1} u_1, P_{N_2} u_2)\|_{L_{T,x}^2}\\
      &\le G(K) T^{\frac{\theta-1}{2}}
       N_1^{\frac{-(1-\theta)(\al-\gamma)}{2}} N_2^{\frac{\theta}{2}}N_2^{-s}
       \Big(T^{-\frac{1}{2}}\|P_{N_1}u_1\|_{L_{T,x}^2}
       +N_1^{-s-\e}\|u_1\|_{L_T^\I H_x^s}\Big)
       \|u_2\|_{L_T^\I H_x^s},
  }
  where $G:\R_{\ge 0}\to \R_{\ge 0}$ is a smooth increasing function depending only on $ f$ and on the Fourier symbol $-i p_{\alpha+1} $ of $ L_{\alpha+1}$.
\end{prop}

The rest of this section is devoted to the proof of Propositions \ref{prop_bistri}--\ref{prop_bistri4}.
We begin by recalling the two bilinear estimates \eqref{bil1} and \eqref{bil2} in Proposition \ref{prop_bil} below.
On one hand, \eqref{eq_bistri2} is based on \eqref{bil2}, and actually its proof is essentially done in Proposition 3.2 of \cite{MT3}.
On the other hand, \eqref{eq_bistri1} relies on both \eqref{bil1} and \eqref{bil2}.
Note that, in sharp contrast to \eqref{eq_bistri2}, we need to take the structure of the nonlinearity into account, and thus we restrict ourselves to nonlinearities that satisfy Hypothesis \ref{hyp2} in \eqref{eq_bistri1}.

\begin{prop}\label{prop_bil}
 Let $\al\in[1,2]$ and $a\in L^\I(\R^2)$ with $ \|a\|_{L^\I} \lesssim 1$.
  Then, there exists $C=C(\xi_0)>0$ such that for any real-valued functions $u_1,u_2 \in L^2(\R^2)$, any $N_1,N_2\in 2^\N \cup\{0\}$ and $L_1,L_2\ge 1$, it holds
  \EQS{\label{bil1}
    \begin{aligned}
      &\|\La_a(Q_{L_1} P_{N_1} u_1, Q_{L_2} P_{N_2} u_2)\|_{L_{t,x}^2}\\
      &\le \frac{C(L_1\wedge L_2)^{\frac 12}(L_1\vee L_2)^{\frac 14}}{\LR{N_1\vee N_2}^{\frac{\al-1}{4}}}
        \|Q_{L_1} P_{N_1} u_1\|_{L_{t,x}^2}\|Q_{L_2} P_{N_2} u_2\|_{L_{t,x}^2}.
    \end{aligned}
  }
  Moreover,
  \EQS{\label{bil2}
    \begin{aligned}
      &\|\La_a(Q_{L_1} P_{N_1} u_1,
        Q_{L_2}P_{N_2}u_2)\|_{L_{t,x}^2}\\
      &\le C(L_1\wedge L_2)^{\frac 12}
      \min \left(\frac{(L_1\vee L_2)^{\frac 12}}{(N_1\vee N_2)^{\frac{\al}{2}}}, \LR{N_1\wedge N_2}^{\frac 12} \right)\\
      &\quad\times
        \|Q_{L_1}P_{N_1}u_1\|_{L_{t,x}^2}
        \|Q_{L_2}P_{N_2}u_2\|_{L_{t,x}^2}.
    \end{aligned}
  }
 whenever $ N_1\vee N_2 \ge 8 \LR{N_1\wedge N_2}$.
 \end{prop}

\begin{proof}
  The proof is identical with that of Proposition 3.3 in \cite{MT3}.
\end{proof}

\begin{rem}
  In \cite{MT3}, it was crucial to borrow some estimates from the work of Ionescu-Kenig-Tataru \cite{IKT}, i.e., we used the following estimate:
  \EQQS{
   \|\chi(L t)u\|_{L_{t,x}^2}
   \le C L^{-\frac 12}\|u\|_{X^{0,\frac 12,1}}.
  }
  Here, $X^{0,\frac 12,1}$ is the Bourgain space of Besov type, and $\chi$ is defined in \eqref{defchi}.
  Restricting the estimate on a very short time interval (of length $\sim L^{-1}$), one can handle an unfavorable term that arises when working on the torus.
  This difficulty comes from the fact that one single point has positive measure in $\Z$.
  In this work, we do not need such an estimate since we work on the real line.
\end{rem}

\begin{lem}\label{lem_short2}
  Let $\al\in[1,2]$, $N_1,N_2\in 2^\N\cup\{0\}$.
  Assume that $a\in L^\I(\R^2)$ satisfies $\|a\|_{L^\I}\lesssim 1$.
  Then, it holds that for $u_1,u_2\in X^{0,\frac{3}{8}}$
  \EQS{\label{bil3}
    \|\La_a(P_{N_1}u_1, P_{N_2}u_2)\|_{L_{t,x}^2}
    \lesssim \LR{N_1\vee N_2}^{-\frac{\al-1}{4}}
    \|P_{N_1}u_1\|_{X^{0,\frac{3}{8}}}
    \|P_{N_1}u_1\|_{X^{0,\frac{3}{8}}}.
  }
  Moreover, it holds that for $\theta\in(0,1]$ and $u_1,u_2\in X^{0,\frac{1}{2}-\frac{\theta}{4}}$
  \EQS{\label{bil4}
    \begin{aligned}
      &\|\La_a(P_{N_1}u_1, P_{N_2}u_2)\|_{L_{t,x}^2}\\
      &\lesssim (N_1\vee N_2)^{-\frac{(1-\theta)\al}{2}}
      \LR{N_1\wedge N_2}^{\frac{\theta}{2}}
       \|P_{N_1}u_1\|_{X^{0,\frac{1}{2}-\frac{\theta}{4}}}
       \|P_{N_2}u_2\|_{X^{0,\frac{1}{2}-\frac{\theta}{4}}}
    \end{aligned}
  }
  whenever $ N_1\vee N_2 \ge 8\LR{N_1\wedge N_2}$.
  In the above estimates, the implicit constants do not depend on $u_1,u_2,N_1$ and $N_2$.
\end{lem}

\begin{proof}
  Without loss of generality, we may assume that $N_1\ge N_2$.
  By \eqref{bil1} and decomposing functions $u_1$ and $u_2$ with respect to the modulation, we have
  \EQQS{
    \|\La_a(P_{N_1}u_1, P_{N_2}u_2)\|_{L_{t,x}^2}
    &\le \sum_{L_1,L_2\ge 1}
    \|\La_a(Q_{L_1}P_{N_1}u_1, Q_{L_2} P_{N_2}u_2)\|_{L_{t,x}^2}\\
    &\lesssim \sum_{L_1\ge L_2\ge 1}
    \|\La_a(Q_{L_1}P_{N_1}u_1, Q_{L_2} P_{N_2}u_2)\|_{L_{t,x}^2}\\
    &\lesssim \LR{N_1}^{-\frac{\al-1}{4}}
    \sum_{L_1\ge L_2\ge 1}
    L_1^{\frac{1}{4}}L_2^{\frac{1}{2}}
    \|Q_{L_1}P_{N_1}u_1\|_{L_{t,x}^2}
    \|Q_{L_2}P_{N_2}u_2\|_{L_{t,x}^2}\\
    &\lesssim \LR{N_1}^{-\frac{\al-1}{4}}
    \sum_{L_1\ge L_2\ge 1}\bigg(\frac{L_2}{L_1}\bigg)^{\frac{1}{8}}
    \|Q_{L_1}P_{N_1}u_1\|_{X^{0,\frac{3}{8}}}\|Q_{L_2}P_{N_2}u_2\|_{X^{0,\frac{3}{8}}}\\
    &\lesssim \LR{N_1}^{-\frac{\al-1}{4}}
    \|P_{N_1}u_1\|_{X^{0,\frac{3}{8}}}
    \|P_{N_2}u_2\|_{X^{0,\frac{3}{8}}}.
  }
  Here, in the last inequality, we used the Young inequality.
  Finally the proof of \eqref{bil4}
   is similar to the one of \eqref{bil3} by using \eqref{bil2} instead of \eqref{bil1} and the fact that $\min\{a,b\}\le a^{1-\theta}b^\theta$ for $a,b\ge 0$.
\end{proof}

In the proof of Proposition \ref{prop_bistri}, we would like to take $\theta=0$ in \eqref{bil4}.
For that purpose, we work on Besov type Bourgain spaces $X^{0,\frac{1}{2},1}$ with \eqref{bil2}.

\begin{lem}\label{lem_short3}
  Let $\al\in[1,2]$ and $N_1\vee N_2\ge 8\LR{N_1\wedge N_2}$.
  Assume that $a\in L^\I(\R^2)$ satisfies $\|a\|_{L^\I}\lesssim 1$.
  Then, it holds that for $u_1,u_2\in X^{0,\frac{1}{2},1}$
  \EQS{\label{bil5}
    \|\La_a(P_{N_1}u_1, P_{N_2}u_2)\|_{L_{t,x}^2}
    \lesssim (N_1\vee N_2)^{-\frac{\al}{2}}
    \|P_{N_1}u_1\|_{X^{0,\frac{1}{2},1}}
    \|P_{N_1}u_1\|_{X^{0,\frac{1}{2},1}},
  }
  where the implicit constant does not depend on $u_1,u_2,N_1$ and $N_2$.
\end{lem}

\begin{proof}
  The proof is identical to that of Lemma \ref{lem_short2}, and is therefore omitted.
\end{proof}

Let us now translate \eqref{bil3}, \eqref{bil4} and \eqref{bil5} in terms of bilinear estimates for free solutions of \eqref{eq1}.

\begin{cor}\label{cor_free1}
 Let $a\in L^\I(\R^2)$ be such that $\|a\|_{L^\I}\lesssim 1$, $\ga\in(0,1]$, $ \alpha\in [1,2] $, $N_1,N_2\in 2^\N\cup\{0\}$, $M\ge 1 $, $0<T<1$ and $\vp_1,\vp_2\in L^2(\R)$.
 Suppose that $I\subset \R$ satisfies $|I|\sim M^{-\ga}T$.
 Then it holds
  \EQS{\label{cor1}
    \begin{aligned}
      &\|\La_a(U_\al(t)P_{N_1}\vp_1,
        U_\al(t)P_{N_2}\vp_2 )\|_{L^2(I;L^2)}\\
      &\lesssim T^{\frac 14} M^{-\frac{\ga}{4}}
       \LR{N_1\vee N_2}^{-\frac{\al-1}{4}}
       \|P_{N_1}\vp_1\|_{L^2}\|P_{N_2}\vp_2\|_{L^2}.
    \end{aligned}
  }
  Moreover, for $\theta\in [0,1]$, if $ N_1\vee N_2 \ge 8\LR{N_1\wedge N_2} $, \eqref{cor1} can be improved to
  \EQS{\label{cor2}
    \begin{aligned}
      &\|\La_a(U_\al(t)P_{N_1}\vp_1,
        U_\al(t)P_{N_2}\vp_2)
        \|_{L^2(I;L^2)}\\
      &\lesssim T^{\frac{\theta}{2}} M^{-\frac{\te\ga}{2}} (N_1\vee N_2)^{-\frac{(1-\theta)\al}{2}}
       \LR{N_1\wedge N_2}^{\frac{\theta}{2}}
       \|P_{N_1}\vp_1\|_{L^2}\|P_{N_2}\vp_2\|_{L^2}.
    \end{aligned}
  }
\end{cor}

\begin{proof}
  Following Lemma 3.4 in \cite{MT22}, \eqref{cor1} and \eqref{cor2} with $\theta>0$ follow from \eqref{bil3} and \eqref{bil4}, respectively.
  On the other hand, using the argument of Corollary 3.8 in \cite{MT3}, \eqref{cor2} with $\theta=0$ follows from \eqref{bil5}.
\end{proof}

Note that we always take $ M=N_1\vee N_2 $ when we apply directly  \eqref{cor1}  in the proof of Proposition \ref{prop_bistri}.
On the other hand, in the proof of Proposition \ref{prop_bistri}, we need to apply Proposition \ref{prop_bistri3} with $M\ge N$, which actually uses \eqref{cor2} with $M\ge N_1\vee N_2$.

\begin{rem}
  In the proof of Proposition \ref{prop_bistri3}, we also use a slightly modified version of \eqref{cor2}.
  Under the same assumption as in Corollary \ref{cor_free1}, we have
  \EQS{\label{cor2.2}
    \begin{aligned}
      &\|\La_a(U_\al(t)P_{N_1}\vp_1,
        U_\al(t)P_{\le N_2}\vp_2)
        \|_{L^2(I;L^2)}\\
      &\lesssim T^{\frac{\theta}{2}} M^{-\frac{\te\ga}{2}} (N_1\vee N_2)^{-\frac{(1-\theta)\al}{2}}
      \LR{N_1\wedge N_2}^{\frac{\theta}{2}}
       \|P_{N_1}\vp_1\|_{L^2}\|P_{\le N_2}\vp_2\|_{L^2},
    \end{aligned}
  }
  This estimate follows from the following inequality:
  \EQQS{
    \|\La_a(Q_{L_1} P_{N_1} u_1,
        Q_{L_2}P_{\le N_2}u_2)\|_{L_{t,x}^2}
    &\le C(L_1\wedge L_2)^{\frac 12}
      \min \left(\frac{(L_1\vee L_2)^{\frac 12}}{(N_1\vee N_2)^{\frac{\al}{2}}}, \LR{N_1\wedge N_2}^{\frac 12} \right)\\
    &\quad\times
        \|Q_{L_1}P_{N_1}u_1\|_{L_{t,x}^2}
        \|Q_{L_2}P_{\le N_2}u_2\|_{L_{t,x}^2},
  }
  which can be shown in the same manner as \eqref{bil2}.
\end{rem}

\begin{rem}
  In the proof of Proposition \ref{prop_bistri}, it is essential to use \eqref{cor1} and \eqref{cor2} with $\ga<1$, which allows us to use bilinear estimates iteratively.
  For the proof of Proposition \ref{prop_bistri2}, however, we only use \eqref{cor2} with $\ga=1$, since the gain of regularity is sufficient for our purpose (see also Proposition \ref{prop_bistri4}).
  In contrast, in the previous work \cite{MT3}, we always had $\ga=1$.
  It is also worth mentioning that we have a positive power of $T$ in both \eqref{cor1} and \eqref{cor2}.
  They are helpful in the proof of well-posedness.
\end{rem}

We will rely on Christ-Kiselev type lemmas (\cite{CK01}) in order to apply the Strichartz estimate in the proof of Proposition \ref{prop_bistri}.
For that purpose, we write the integral of a function $f$ with a kernel $K(t,s)$ as
\EQS{\label{def_I}
  \mathcal{I}_K(f;[t_0,t_2])(t)= \int_{t_0}^{t_2} K(t,s)f(s)ds.
}
The following version is taken from the proof of Theorem 1.21 in \cite{Tzv}.

\begin{lem}\label{lem_ck}
  Let $t_0, t_1,t_2\in\R$ with $t_1<t_2$ and $t_1\le t_0\le t_2$, and let $1\le p<q\le \I$.
  Let $X$ and $Y$ be Banach spaces and assume that $K(t,s)$ is a continuous function taking its values in $B(X,Y)$, the space of bounded linear operators from $X$ to $Y$.
  Assume that for some $C_1>0$ it holds for $f\in L^p([t_1,t_2];X)$,
  \EQQS{
    \|\mathcal{I}_K(f;[t_1,t_2])(t)\|_{L^q([t_1,t_2];Y)}
    \le C_1 \|f\|_{L^p([t_1,t_2];X)}.
  }
  Then, there exists $C_2>0$ depending only on $p$ such that for any $f\in L^p([t_1,t_2];X)$,
  \EQQS{
    \|\mathcal{I}_K(f;[t_0,t])(t)\|_{L^q([t_1,t_2];Y)}
    \le C_1C_2 \|f\|_{L^p([t_1,t_2];X)}.
  }
\end{lem}

We also state the bilinear versions of Lemma \ref{lem_ck}, which play a fundamental role in our analysis.
The reason we formulate these estimates using a symbol $a$ is that, we intend to apply integration by parts and the bilinear estimates simultaneously; see Case (3) in the proof of Proposition \ref{prop_apri1}.
These two results can be derived by the same argument as used in Lemma \ref{lem_ck}.

\begin{cor}\label{cor_CK1}
  Let $t_0,t_1,t_2\in\R$ with $t_1<t_2$ and $t_1 \le t_0\le t_2$, and let $1\le p<2$ and $1\le q\le \I$.
  Let $a\in L^\I(\R^2)$ satisfy $\|a\|_{L^\I}\les 1$.
  Assume that $K(t,s)$ is a continuous function taking values in $B(L^2)$, the space of bounded linear operators on $L^2$.
  Suppose that for some $C_1>0$, it holds for any $f\in L^{p}([t_1,t_2];L^{q})$ and $g\in L^2$ that
  \EQQS{
    \big\|\La_a\big(\mathcal{I}_K(f;[t_1,t_2])(t),K(t,0) g\big) \big\|_{L^2([t_1,t_2];L^2)}
    \le C_1\|f\|_{L^{p}([t_1,t_2];L^{q})}\|g\|_{L_x^2}.
  }
  Then, there exists $C_2>0$ depending only on $p$ such that it holds for any $f\in L^{p}([t_1,t_2];L^{q})$ and $g\in L^2$,
  \EQQS{
    \big\|\La_a\big(\mathcal{I}_K(f;[t_0,t])(t),K(t,0) g\big) \big\|_{L^2([t_1,t_2];L^2)}
    \le C_1C_2
     \|f\|_{L^{p}([t_1,t_2];L^{q})}\|g\|_{L_x^2}.
  }
\end{cor}

\begin{cor}\label{cor_CK2}
  Let $t_0,t_1,t_2,p,q\in\R$, $a\in L^\I(\R^2)$, and $K\in B(L^2)$ satisfy the same assumption as in Corollary \ref{cor_CK1}.
  Assume that there exists $C_1>0$ such that for any $f\in L^{p_1}([t_1,t_2];L^{q_1})$ and $g\in L^{p_2}([t_1,t_2];L^{q_2})$, it holds that
  \EQQS{
    &\big\|\La_a\big(\mathcal{I}_K(f;[t_1,t_2])(t),
     \mathcal{I}_K(g;[t_1,t_2])(t)\big) \big\|_{L^2([t_1,t_2];L^2)}\\
    &\le C_1\|f\|_{L^{p_1}([t_1,t_2];L^{q_1})}
     \|g\|_{L^{p_2}([t_1,t_2];L^{q_2})}.
  }
  Then, there exists $C_2>0$ depending only on $p_1,p_2$ such that for any $f\in L^{p_1}([t_1,t_2];L^{q_1})$ and $g\in L^{p_2}([t_1,t_2];L^{q_2})$, it holds that
  \EQQS{
    &\big\|\La_a\big(\mathcal{I}_K(f;[t_0,t])(t),
     \mathcal{I}_K(g;[t_0,t])(t)\big) \big\|_{L^2([t_1,t_2];L^2)}\\
    &\le C_1C_2
    \|f\|_{L^{p_1}([t_1,t_2];L^{q_1})}
    \|g\|_{L^{p_2}([t_1,t_2];L^{q_2})},
  }
  provided that $1\le p_1,p_2<2$ and $1\le q_1,q_2\le \I$.
\end{cor}

The following estimate plays a key role in the proof of Proposition \ref{prop_bistri}.

\begin{prop}\label{prop_bistri3}
  Let $0<T<1$, $\al\in[1,2]$, $\ga=\max\{1-\al/3,2-\al\}$,  $M\ge N\ge 1$, $s>1/2$, and $0<\e\ll 1$ so that $s>1/2+\e$.
  Let $I\subset [0,T]$ be a closed interval of length $|I|\sim M^{-\ga}T$.
  Assume that $u\in L^\I (]0,T[;H^s(\R))$ satisfies \eqref{eq1} on $]0,T[\times \R$ and
  $\|u\|_{L_T^\I H_x^{\frac{1}{2}+\e}}\le K$ for some $K\ge 1$.
  Then, for sufficiently small $T$ satisfying $G(K)T^{\frac{1}{2}}\ll 1$, it holds that
  \EQS{\label{eq_sep1}
    \begin{aligned}
      &\|P_{N}u P_{\le 2^{-5} N}u\|_{L^2(I;L^2)}\\
      &\le G(K) N^{-\frac{\al}{2}}
       \left(T^{-\frac{1}{2}}M^{\frac{\ga}{2}}
       \|P_N u\|_{L^2(I;L^2)}
       +N^{-s-\e}\|u\|_{L_T^\I H_x^s}\right),
    \end{aligned}
  }
 where $G:\R_{\ge 0}\to \R_{\ge 0}$ is a smooth increasing function depending only on $ f$ and on the Fourier symbol $-i p_{\alpha+1} $ of $ L_{\alpha+1}$.
\end{prop}

\begin{rem}\label{rem_sep1}
  Let $\{\om_N^{(\de)}\}$ be an acceptable frequency weight.
  Proceeding exactly as in the proof of \eqref{eq_sep1}, we obtain
  \EQQS{
    &\|P_{N}u P_{\le 2^{-5} N }u\|_{L^2(I;L^2)}\\
    &\le G(K) N^{-\frac{\al}{2}}
     \left(T^{-\frac{1}{2}}M^{\frac{\ga}{2}}\|P_N u\|_{L^2(I;L^2)}
     +\om_N^{-1}N^{-s-\e}\|u\|_{L_T^\I H_\om^s}\right).
  }
  For the same reason, we can also write \eqref{eq_bistri1} with frequency envelopes as
  \EQQS{
    &\| \La_a(P_{N_1} u_1, P_{N_2} u_2)\|_{L_{T,x}^2}\\
    &\le G(K) T^{\frac{1}{4}}
     N_1^{\frac{\ga+1-\al}{4}}N_2^{-s}
     \left(\|P_{N_1}u_1\|_{L_{T,x}^2}
     +\om_{N_1}^{-1}N_1^{-s-\e}\|u_1\|_{L_T^\I H_\om^s}\right)
     \om_{N_2}^{-1}\|u_2\|_{L_T^\I H_\om^s}.
  }
\end{rem}

\begin{proof}[Proof of Proposition \ref{prop_bistri3}]
  It suffices to consider the case $N\ge 2^6$.
  We first notice that, since $u\in L^\infty_T H^{\frac12+}$, we have $ P_N u P_{\ll 2^{-5}N} u \in L^2_T L^2_x $, and thus the quantity on the left-hand side of \eqref{eq_sep1} is well-defined.
  Put $N(k):=N/(16k)$ for $k\in \N$.
  First, we rewrite the equation as follows:
  \EQS{\label{eq2}
    \p_t u+ (L_{\al+1}+f'(0)\p_x)u+\p_x (\tilde{f}(u))=0,
  }
  where $\tilde{f}(x)=f(x)-f(0)-f'(0)x$ for $x\in \R$.
  It is important to observe that the operator $L_{\al+1}+f'(0)\p_x$ still satisfies Hypothesis \ref{hyp1}, so that the Strichartz estimates \eqref{stri_1} and \eqref{stri_2} remain applicable.
  For this reason, we continue to use the notaion $U_\al(t)$ for the linear propagator associated with $L_{\al+1}+f'(0)\p_x$.
  We may set $I=[t_1,t_2]$ for $t_1,t_2\in\R$ with $t_1<t_2$.
  It follows from Remark \ref{rem_conti} that $u\in C([0,T];H^{s-}(\R))$.
  By this continuity, there exists $t_0\in I$ at which $\|P_N u(t)\|_{L_x^2}$ attains its minimum over $I$, so that
  \EQS{\label{eq3.16}
    \|P_N u(t_0)\|_{L_x^2}
    \le \frac{1}{|I|}\int_I\|P_N u(t')\|_{L_x^2}dt'
    \le T^{-\frac{1}{2}}M^{\frac{\ga}{2}}
     \|P_N u\|_{L^2(I;L_x^2)}.
  }
  By the Duhamel principle, we have
  \EQQS{
    P_N u(t)
    = U_\al(t-t_0) P_N u(t_0)
     -\int_{t_0}^t U_\al(t-t')\p_x P_N \tilde{f}(u) dt'
  }
  on $I$.
  The triangle inequality shows that
  \EQQS{
    &\|P_N u P_{\le N(2)} u \|_{L^2(I;L^2)}\\
    &\le \|U_\al(t-t_0) P_N u(t_0) U_\al(t-t_0) P_{\le N(2)} u(t_0) \|_{L^2(I;L^2)}\\
    &\quad+\bigg\| U_\al(t-t_0) P_N u(t_0) \left(\int_{t_0}^t U_\al(t-t')\p_x P_{\le N(2)} \tilde{f}(u) dt'\right)  \bigg\|_{L^2(I;L^2)}\\
    &\quad+\bigg\| \bigg(\int_{t_0}^t U_\al(t-t')\p_x P_N \tilde{f}(u) dt'\bigg)
     U_\al(t-t_0) P_{\le N(2)} u(t_0) \bigg\|_{L^2(I;L^2)}\\
    &\quad+\bigg\| \bigg(\int_{t_0}^t U_\al(t-t')\p_x P_N \tilde{f}(u) dt'\bigg)
     \int_{t_0}^t U_\al(t-t')\p_x P_{\le N(2)} \tilde{f}(u) dt'  \bigg\|_{L^2(I;L^2)}\\
    &=:A_1+A_2+A_3+A_4.
  }
  By \eqref{cor2.2} and \eqref{eq3.16}, we have
  \EQQS{
    A_1
    \les N^{-\frac{\al}{2}}\|P_N u(t_0)\|_{L_x^2}
     \|P_{\le N(2)} u(t_0)\|_{L_x^2}
    \les KT^{-\frac{1}{2}}N^{-\frac{\al}{2}}M^\frac{\ga}{2}
    \|P_N u\|_{L^2(I;L^2)}.
  }
  As for $A_2$, we see from \eqref{cor2} and \eqref{stri_2} with $\te=1$ that
  \EQS{\label{eq3.17XX}
    \begin{aligned}
      &\bigg\| U_\al(t-t_0) P_{N} h_1
       \bigg(\int_{t_0}^{t_2} U_\al(t-t')\p_x P_{N_1} h_2 dt'\bigg)  \bigg\|_{L^2(I;L^2)}\\
      &\lesssim N^{-\frac{\al}{2}} \|P_{N} h_1\|_{L_x^2}
       \bigg\| \int_{t_0}^{t_2} U_\al(-t')\p_x P_{N_1} h_2 dt'  \bigg\|_{L_x^2}\\
      &\lesssim N^{-\frac{\al}{2}} \LR{N_1}^{\frac{5-\al}{4}} \|h_1\|_{L_x^2}
       \|h_2\|_{L^{\frac{4}{3}}(I;L^1)}
    \end{aligned}
  }
  for $N_1\ll N$.
  Moreover, we see from the Leibniz rule \eqref{leibniz2} that
  \EQS{\label{eq3.17}
    \begin{aligned}
      \LR{N_1}^{\frac{5-\al}{4}}
       \|P_{N_1} \tilde{f}(u)\|_{L^{\frac{4}{3}}(I;L^1)}
      &\les T^\frac{3}{4} M^{-\frac{3}{4}\ga}
        \LR{N_1}^{\frac{5-\al}{4}} \LR{N_1}^{-\frac12-\e}
       \|J_x^{\frac{1}{2}+\e}\tilde{f}(u)\|_{L_T^\I L_x^1} \\
      &\les T^{\frac{3}{4}}
       \LR{N_1}^\frac{3-\alpha-3\gamma}{4}
       N^{-\e} G(K),
    \end{aligned}
  }
  since  $|I|\sim M^{-\ga}T$, $M\ge N\ge 1$, and $N\gg N_1$.
  Then, \eqref{eq3.16}, \eqref{eq3.17}, Corollary \ref{cor_CK1} and \eqref{eq3.17XX} with $h_1= u(t_0)$, $h_2=\tilde{f}(u)$ and $ N_1=N(2)$  imply that
  \EQQS{
    A_2
    &\le  \bigg\| U_\al(t-t_0) P_N u(t_0)
     \bigg(\int_{t_0}^{t} U_\al(t-t')
      \p_x P_{\le N(2)} \tilde{f}(u) dt'\bigg)  \bigg\|_{L^2(I;L^2)}\\
    &\lesssim N^{-\frac{\al}{2}} \|P_N u(t_0)\|_{L_x^2}
      \sum_{N_1\le N(2)}
      \LR{N_1}^{\frac{5-\alpha}{4}}
     \|P_{N_1}\tilde{f}(u)\|_{L^{\frac{4}{3}}(I;L^1)}\\
    &\lesssim G(K) T^{\frac{1}{4}}N^{-\frac{\al}{2}-\e}
     M^\frac{\ga}{2}
     \|P_N u\|_{L^2(I;L^2)},
  }
  assuming that $3-\alpha-3\gamma \le 0 $.
  Let us now tackle the estimate of   $A_3$.
  We first notice that, in the same way as for \eqref{eq3.17}, applying \eqref{cor2.2} and \eqref{stri_2} with $\te\in \{0,1\} $ it holds  for $ N\gg N_1\ge 1$,
   \EQS{\label{toto}
    \begin{aligned}
      &\bigg\|
       \bigg(\int_{t_0}^{t_2}  U_\al(t-t')\p_x P_{N} h_1 dt'\bigg)
       U_\al(t-t_0) P_{\le N_1} h_2\bigg\|_{L^2(I;L^2)}\\
      &\lesssim  N^{-\frac{\al}{2}} N
       N^{\frac{1-\alpha}{4} \theta}
       \|h_1\|_{L^{\frac{4}{4-\theta}}(I;L^{\frac{2}{\theta+1}})}\|h_2\|_{L_x^2}.
    \end{aligned}
  }

  Now we decompose $P_N\tilde{f}(u)$.
  By Hypothesis \ref{hyp2}  we have
  \EQS{\label{decom1}
    \begin{aligned}
      P_N\tilde{f}(u)
      &=\sum_{k=2}^{N(2)} \frac{f^{(k)}(0)}{k!}P_N(u^k)
       +P_N f_{>N(2)}(u),
    \end{aligned}
  }
  where $f_{>N(2)}(x) = \tilde{f}(x)-\sum_{k=2}^{N(2)}\frac{f^{(k)}(0)}{k!}x^k$ and by the  Littlewood-Paley decomposition, we can rewrite, for $ k\in \{2,\dots,N(2)\}$,
  \begin{equation}\label{trz}
     P_N(u^k)
     =\sum_{j=1}^k\binom{k}{j}
      P_N\left((P_{> N(k)}u)^j (P_{\le N(k)}u)^{k-j}\right) \, .
  \end{equation}
  Recall that $N(k)=N/(16k)$.
  Here, note that the term corresponding to $j=0$ in \eqref{trz} vanishes automatically due to an impossible interaction.
  We aim to apply \eqref{toto} with either $\theta=0 $ or $\theta=1 $ to estimate the contribution of the different  terms in \eqref{trz}.
  A key point in our argument is that, on the one hand, for $ j\ge 2$, we can place half a derivative on one of the $u$'s, since at least two functions have frequencies of modulus $\gtrsim N(k)$.
  On the other hand, for $j=1$, we can extract the term $\|P_N u P_{\le N(2)} u\|_{L^2_I L^2_x} $ on the right hand side of \eqref{eq_sep1}, along with some commutator terms for which we can also recover at least half a derivative.
  This will lead to a suitable estimate in which we succeed in obtaining a small factor in front of $\|P_N u P_{\le N(2)} u\|_{L^2_I L^2_x} $.
  For that purpose, we observe the following: first notice that $N(k)\ge 2$ when $k=2,\dots, N(2)$.
  For the term corresponding to $j=1$ of the first term in \eqref{decom1}, we have
  \EQQS{
    |\xi_1|\ge |\xi|-|\xi_2|-\cdots-|\xi_k|
    \ge \frac{N}{2}-\frac{(k-1)N}{8k}\ge \frac{N}{4},
  }
  where $\xi=\sum_{j=1}^k \xi_k$, $N/2\le |\xi|\le 2N$ and $|\xi_j|\le N/8k$ for $j=2,\dots,k$.
  Similarly, it is easy to see that $|\xi_1|\le 4N$.
  This implies that
  \EQS{\label{decom2}
    \begin{aligned}
      &P_N(u^k)\\
      &=kP_{\sim N}(P_{N}u(P_{\le N(k)}u)^{k-1})+kP_{\sim N}\left( [P_N,(P_{\le N(k)}u)^{k-1}]P_{\sim N} u\right)\\
      &\quad+ \sum_{j=2}^k\binom{k}{j}
       P_N((P_{> N(k)}u)^j (P_{\le N(k)}u)^{k-j})\\
      &=kP_{\sim N}(P_{N}uP_{\le N(2)}u(P_{\le N(k)}u)^{k-2})
       -kP_{\sim N}(P_{N}uP_{[N(2),N(k)]}u(P_{\le N(k)}u)^{k-2})\\
      &\quad+k P_{\sim N} \left( [P_N,(P_{\le N(k)}u)^{k-1}]P_{\sim N}u\right)
       +\sum_{j=2}^k\binom{k}{j}
       P_N((P_{> N(k)}u)^j (P_{\le N(k)}u)^{k-j})\\
      &=:P_{\sim N}V_1+P_{\sim N}V_2+P_{\sim N}V_3
       +\sum_{j=2}^k P_{ N} V_{j+2},
    \end{aligned}
  }
  where we set $P_{[N(2),N(k)]}u=(P_{\le N(2)}-P_{\le N(k)})u$ and  we used $P_{\sim N}P_N=P_N$.
  Coming back to $A_3$, we further decompose $A_3$ as
  \EQQS{
    A_3
    &\le \sum_{k=2}^{N(2)}\sum_{j=1}^{k+2} \frac{|f^{(k)}(0)|}{k!}
     \bigg\| \bigg(\int_{t_0}^t U_\al(t-t')\p_xP_{\sim N}
     V_j dt'\bigg)
     U_\al(t-t_0) P_{\le N(2)} u(t_0) \bigg\|_{L^2(I;L^2)}\\
    &\quad+
     \bigg\| \bigg(\int_{t_0}^t U_\al(t-t')\p_x
      P_N f_{>N(2)}(u) dt'\bigg)
      U_\al(t-t_0) P_{\le N(2)} u(t_0) \bigg\|_{L^2(I;L^2)}\\
   &=:\sum_{k=2}^{N(2)}\sum_{j=1}^{k+2} \frac{|f^{(k)}(0)|}{k!}A_{3,1,j}+A_{3,2}.
  }
    Applying \eqref{toto} with $ \theta=0 $ together with  the Christ-Kiselev lemma (Corollary \ref{cor_CK1}), we eventually get
  \EQQS{
    &\sum_{k=2}^{N(2)}\frac{|f^{(k)}(0)|}{k!}A_{3,1,1}\\
    &\les \sum_{k=2}^{N(2)}\frac{|f^{(k)}(0)|}{(k-1)!}
     N^{1-\frac{\al}{2}}
     \|P_N u P_{\le N(2)} u (P_{\le N(k)}u)^{k-2}\|_{L^1(I;L^2)}
     \|P_{\le N(2)}u(t_0)\|_{L_x^2}\\
    &\le N^{1-\frac{\al}{2}}M^{-\frac{\ga}{2}}
     G(K)T^{\frac{1}{2}}
     \|P_N u P_{\le N(2)}u\|_{L^2(I;L^2)},
  }
  where $G:\R_{\ge 0}\to\R_{\ge 0}$ is a smooth increasing function.
  Here, it is important to observe that $N^{1-\frac{\al}{2}}M^{-\frac{\ga}{2}}\le N^{1-\frac{\al}{2}-\frac{\ga}{2}}\le1$, since $M\ge N$ and $\ga\ge 2-\al$.
  Next, we consider $A_{3,1,2}$, where we can place the derivative on $P_{[N(2),N(k)]}u=(P_{N(2)}-P_{N(k)})u$.
  Indeed, it is easy to see that
  \EQS{\label{eq3.27}
    \begin{aligned}
      \|V_2\|_{L^{\frac{4}{3}}(I ;L^1)}
      &\les k K^{k-2} |I|^{\frac{3}{4}}
       \|P_N u\|_{L_T^\I L_x^2}
       \|P_{[N(2),N(k)]}u\|_{L_T^\I L_x^2}\\
      &\les k(k-1) K^{k-1} T^{\frac{3}{4}}M^{-\frac{3}{4}\ga}
       N^{-\frac{1}{2}-s-\e}
       \|u\|_{L_T^\I H_x^s},
    \end{aligned}
  }
  where we used $N(k)^{-\frac{1}{2}-\e}\le 2^5(k-1)N^{-\frac{1}{2}-\e}$.
  Therefore, we see from \eqref{eq3.27} and  \eqref{toto} with $\te=1$ and the Christ-Kiselev lemma (Corollary \ref{cor_CK1}) that
  \EQQS{
    \sum_{k=2}^{N(2)}\frac{|f^{(k)}(0)|}{k!}A_{3,1,2}
    &\les \sum_{k=2}^{N(2)}\frac{|f^{(k)}(0)|}{k!}
     N^{-\frac{\al}{2}}
     N^{1-\frac{\al-1}{4}}
     \|V_2\|_{L^{\frac{4}{3}}(I ;L^1)}
     \|P_{\le N(2)}u(t_0)\|_{L_x^2}\\
    &\le N^{-\frac{\al}{2}}
     N^{-s-\e} N^\frac{3-\alpha-3\gamma}{4} G(K) T^{\frac{3}{4}}\|u\|_{L_T^\I H_x^s} \\
       &\lesssim G(K) T^{\frac{3}{4}}  N^{-\frac{\al}{2}-s-\e} \|u\|_{L_T^\I H_x^s},
  }
  since $M\ge N$ and $3-\al-3\ga\le 0$.
  As for $A_{3,1,3}$, we use \eqref{eq_comm2} to handle the commutator, so that
  \EQS{\label{eq3.21}
    \begin{aligned}
      \|V_3\|_{L^{\frac{4}{3}}(I;L^1)}
      &\lesssim k |I|^{\frac{3}{4}}
       \|[P_{N},(P_{\le N(k)}u)^{k-1}]P_{\sim N}u\|_{L_T^\I L_x^1}\\
      &\lesssim k T^{\frac{3}{4}}M^{-\frac{3}{4}\ga}N^{-1+\frac{1}{2}-\e}
       \|P_{\sim N}u\|_{L_T^\I L_x^2}
       \|J_x^{\frac{1}{2}+\e} (P_{\le N(k)}u)^{k-1} \|_{L_T^\I L_x^2}\\
      &\lesssim k K^{k-1} T^{\frac{3}{4}}
       N^{-\frac{1}{2}-\frac{3}{4}\ga-s-\e}
       \|u\|_{L_T^\I H_x^s}\, .
    \end{aligned}
  }
Again   \eqref{toto} with $\te=1$ and Christ-Kiselev lemma (Corollary \ref{cor_CK1}) then lead to
  \EQQS{
    \sum_{k=2}^{N(2)}\frac{|f^{(k)}(0)|}{k!}A_{3,1,3}
    &\lesssim \sum_{k=2}^{N(2)}\frac{|f^{(k)}(0)|}{k!}
     N^{-\frac{\al}{2}} N^{\frac{5-\al}{4}}
     \|V_3\|_{L^{\frac{4}{3}}(I;L^1)}
     \|P_{\le N(2)}u(t_0)\|_{L_x^2}\\
    &\le N^{-\frac{\al}{2}} N^{-s-\e} G(K)T^{\frac{3}{4}}
     \|u\|_{L_T^\I H_x^s},
  }
  since $3-\al-3\ga\le 0$.
  As for $A_{3,1,j+2}$ with $j=2,\dots,k$, we can apply the argument in $A_{3,1,2}$, i.e., allowing us to distribute half a derivative to one of them.
  That is, it holds that
  \EQS{\label{eq3.22}
    \begin{aligned}
      &\|P_{N}((P_{> N(k)}u)^{j}
       (P_{\le N(k)}u)^{k-j})\|_{L^\frac{4}{3}(I;L^1)}\\
      &\les T^{\frac{3}{4}}M^{-\frac{3}{4}\ga}K^{k-j}
       \|(P_{>N(k)}u)^{j-1}\|_{L_T^\I L_x^2}
       \sum_{M_1\ge N(k)}
       \|P_{M_1}u\|_{L_T^\I L_x^2}\\
      &\les k^{s+1} T^{\frac{3}{4}}
       N^{-\frac{1}{2}-\frac{3}{4}\ga-s-\e}
       K^{k-1}\|u\|_{L_T^\I H_x^s},
    \end{aligned}
  }
  since $N(k)^{-\frac{1}{2}-s-\e}\les k^{s+1} N^{-\frac{1}{2}-s-\e}$.
  Again, \eqref{toto} with $\te=1$ and Corollary \ref{cor_CK1} together with \eqref{eq3.22} lead to
  \EQQS{
    \sum_{k=2}^{N(2)}\sum_{j=2}^{k}\frac{|f^{(k)}(0)|}{k!}A_{3,1,j+2}
    &\lesssim \sum_{k=2}^{N(2)}\sum_{j=2}^{k}\frac{|f^{(k)}(0)|}{k!}
      N^{-\frac{\al}{2}}
      N^{\frac{5-\al}{4}}
     \|V_{j+2}\|_{L^{\frac{4}{3}}(I;L^1)}
     \|P_{\le N(2)}u(t_0)\|_{L_x^2}\\
    &\lesssim \sum_{k=2}^{N(2)}\frac{|f^{(k)}(0)|}{(k-1)!}
     N^{-\frac{\al}{2}} N^{-s-\e} N^\frac{3-\alpha-3\gamma}{4}
     2^k K^k T^{\frac{3}{4}}\|u\|_{L_T^\I H_x^s}\\
    &\les N^{-\frac{\al}{2}} N^{-s-\e}
     G(2K)T^{\frac{3}{4}}\|u\|_{L_T^\I H_x^s}.
  }
  Finally, we estimate $A_{3,2}$.
  By the product estimate, it is easy to see that
  \EQS{\label{eq3.26}
    \|f_{>N(2)}(u)\|_{H_x^s}
    \les \sum_{k=N(2)+1}^\I \frac{|f^{(k)}(0)|}{k!}
     \|u\|_{L_x^\I}^{k-1}\|u\|_{H_x^s}
    \les N^{-1}G(K)\|u\|_{H_x^s},
  }
  since $N(2)=2^{-5}N$ and $s>1/2$.
  This together with \eqref{toto} with $\te=0$ implies that
  \EQQS{
    A_{3,2}
    &\les N^{1-\frac{\al}{2}}\|P_N f_{>N(2)}(u)\|_{L^1(I;L^2)}
     \|P_{\le N(2)}u(t_0)\|_{L_x^2}\\
    &\les K TN^{1-\frac{\al}{2}} N^{-s} M^{-\ga}
     \|f_{>N(2)}(u)\|_{L_T^\I H_x^s}
    \les TN^{-\frac{\al}{2}} N^{-s-\ga}G(K)\|u\|_{L_T^\I H_x^s}
  }
  In the same way, the estimates for $A_4$ follows from those for $A_3$, using \eqref{eq3.17} and Corollary \ref{cor_CK2} instead of Corollary \ref{cor_CK1}.
  Therefore, collecting the estimates above, we get
  \EQQS{
    &\|P_N u P_{\le N(2)}u\|_{L^2(I;L^2)}\\
    &\le CN^{-\frac{\al}{2}} (K+T^{\frac{3}{4}}G(K))
    \left(M^{\frac{\ga}{2}} T^{-\frac{1}{2}}\|P_N u\|_{L^2(I;L^2)}
     +N^{-s-\e}G(K)T^{\frac{3}{4}}\|u\|_{L_T^\I H_x^s}\right)\\
    &\quad+CN^{1-\frac{\al}{2}-\frac{\ga}{2}}G(K)T^{\frac{1}{2}}
     \|P_N u P_{\le N(2)}u\|_{L^2(I;L^2)}.
  }
  Since $2-\al-\ga\le 0$, by choosing $T$ such that $CG(K)T^{\frac{1}{2}}\ll 1$, we obtain \eqref{eq_sep1}.
  This concludes the proof.
\end{proof}

\begin{proof}[Proof of Proposition \ref{prop_bistri}]
  When $N_1\lesssim 2^6$, the Young inequality with $\|a\|_{L^\I}\les 1$ and the Bernstein inequality show that
  \EQS{\label{eq3.28}
    \begin{aligned}
        \|\La_a(P_{N_1}u_1,P_{N_2}u_2)\|_{L_{T,x}^2}
      &\les N_2^{\frac{1}{2}} \|P_{N_1}u_1\|_{L_{T,x}^2}
        \|P_{N_2}u_2\|_{L_T^\I L_x^2}\\
      &\les N_1^{\frac{\ga+1-\al}{4}} \|P_{N_1}u_1\|_{L_{T,x}^2}
        \|P_{N_2}u_2\|_{L_T^\I L_x^2}.
    \end{aligned}
  }
  Therefore, we only consider the case $N_1\gg 2^6$.
  We divide the time interval $[0,T]$ into closed subintervals of length $\sim N_1^{-\ga} T $;
  that is, we define $\{I_{j,N_1}\}_{j\in J_{N_1}}$ with $  \# J_{N_1}\sim N_1^\ga $ such that $\bigcup_{j\in J_{N_1}}I_{j,N_1}=[0,T]$ and $|I_{j,N_1}|\sim N_1^{-\ga} T$.
  Without loss of generality, we may set $I_{j,N_1}=[t_j,t_{j+1}]$ for $j\in J_{N_1}$.
  Remark \ref{rem_conti} implies that $u\in C([0,T];H^{s-}(\R))$.
  By this continuity, there exists $\tilde{t}_{j,N_1}\in I_{j,N_1}$ at which $\|P_{N_1}u_1(t)\|_{L_x^2}$ attains its minimum on $I_{j,N_1}$.
  For simplicity, we write $t_j=t_{j,N_1}, \tilde{t}_j=\tilde{t}_{j,N_1}$, and $N_m(k)=N_m/(16k)$ for $m=1,2$ and $k\in\N$.
  Since $u_1$ and $u_2$ satisfy \eqref{eq2} on $ I_{j,N_1} $, it holds for $m=1,2$
  \EQQS{
    P_{N_m}u_m(t)
    &=U_\al(t-\tilde{t}_j)P_{N_m} u_m(\tilde{t}_j)+\int_{\tilde{t}_j}^t U_\al(t-\ta) \p_x P_{N_m} \tilde{f}(u_m) d\ta\\
    &=:U_\al(t-\tilde{t}_j)P_{N_m} u_m(\tilde{t}_j) +F_{m,j}.
  }
  Therefore
  \EQQS{
    &\|\La_a (P_{N_1} u_1, P_{N_2} u_2)\|_{L_{T,x}^2}^2\\
    &=\sum_{j\in J_{N_1}}
     \|\La_a (P_{N_1} u_1, P_{N_2} u_2)\|_{L^2(I_{j,N_1};L^2)}^2\\
    &\le \sum_{j\in J_{N_1}}
     \|\La_a(P_{N_1}U_\al(t-\tilde{t}_j)u_1(\tilde{t}_j),
     P_{N_2}U_\al(t-\tilde{t}_j)u_2(\tilde{t}_j))\|_{L^2(I_{j,N_1};L^2)}^2\\
    &\quad+\sum_{j\in J_{N_1}}A_{j}^2
     +\sum_{j\in J_{N_1}}B_{j}^2
     +\sum_{j\in J_{N_1}} D_{j}^2,
  }
  where
  \EQQS{
    A_{j}&=\|\La_a(F_{1,j},
     P_{N_2}U_\al(t-\tilde{t}_j)u_2(\tilde{t}_j))\|_{L^2(I_{j,N_1};L^2)},\\
    B_{j}&=\|\La_a(P_{N_1}U_\al(t-\tilde{t}_j)u_1(\tilde{t}_j),
     F_{2,j})\|_{L^2(I_{j,N_1};L^2)},\\
    D_{j}&=\|\La_a(F_{1,j},F_{2,j})\|_{L^2(I_{j,N_1};L^2)}.
  }
  We see from \eqref{cor1} and \eqref{eq3.16} that
  \EQS{\label{eq3.23}
    \begin{aligned}
      &\sum_{j\in J_{N_1}} \|\La_a(P_{N_1}U_\al(t-\tilde{t}_j)u_1(\tilde{t}_j),
       P_{N_2}U_\al(t-\tilde{t}_j)u_2(\tilde{t}_j))\|_{L^2(I_{j,N_1};L^2)}^2\\
      &\lesssim \sum_{j\in J_{N_1}}  T^{\frac 12} N_1^{-\frac{\al+\ga-1}{2}}
        \|P_{N_1} u_1(\tilde{t}_j)\|_{L_x^2}^2
        \|P_{N_2} u_2(\tilde{t}_j)\|_{L_x^2}^2\\
      &\lesssim T^{-\frac 12}N_1^{\frac{\ga+1-\al}{2}}
      \|P_{N_1} u_1\|_{L^2_{T,x}}^2\|P_{N_2} u_2\|_{L_T^\I L^2_x}^2.
    \end{aligned}
  }
  For the contribution of $A_{j}$ for $j\in J_{N_1}$, first we notice that \eqref{cor1} implies that for $\theta\in \{0,1\} $, $h_1\in L^{\frac{4}{4-\theta}}(I_{j,N_1};L^{\frac{2}{\theta+1}})$ and $h_2\in L_x^2$,
  \EQS{\label{eq3.24}
    \begin{aligned}
      &\bigg\|\La_a\bigg(\int_{a_{j}}^{b_{j}}
       U_\al(t-\ta) \p_x P_{\sim N_1} h_1 d\ta,
       U_\al(t)P_{\sim N_2} h_2\bigg)\bigg\|_{L^2(I_{j,N_1};L^2)}\\
      &\lesssim T^{\frac{1}{4}}N_1^{-\frac{\al+\ga-1}{4}}
       \bigg\|\int_{a_{j}}^{b_{j}}
        U_\al(-\ta)\p_x P_{\sim N_1}h_1d\ta\bigg\|_{L_x^2}
       \|h_2\|_{L_x^2}\\
      &\lesssim T^{\frac{1}{4}}N_1^{-\frac{\al+\ga-1}{4}}
       N_1^{1-\frac{\al-1}{4}\te}
       \|h_1\|_{L^{\frac{4}{4-\theta}}(I_{j,N_1};L^{\frac{2}{\theta+1}})}
       \|h_2\|_{L_x^2},
    \end{aligned}
  }
  where   $I_{j,N_1}=[t_j,t_{j+1}]$.
  Then, in light of the decompositions \eqref{decom1} and \eqref{decom2}, we see from \eqref{eq3.24} with $h_1=V_m$ or $P_{N_1}f_{>N_1(2)}(u_1)$ and $h_2=U_\al(-\tilde{t}_j)P_{N_2}u_2(\tilde{t}_j)$ that
  \EQS{\label{eq3.25}
    \begin{aligned}
    A_j
    &\les \sum_{k=2}^{N_1(2)}
     \frac{|f^{(k)}(0)|}{k!} \sum_{m=1}^{k+2} T^{\frac{1}{4}}
     N_1^{-\frac{\al+\ga-1}{4}} N_1^{1-\frac{\al-1}{4}\te_m}
     \|V_m\|_{L^{\frac{4}{4-\theta_m}}(I_{j,N_1};L^{\frac{2}{\theta_m+1}})}\|P_{N_2}u_2(\tilde{t}_j)\|_{L_x^2}\\
    &\quad+T^{\frac{1}{4}}
    N_1^{-\frac{\al+\ga-1}{4}+1}
    \|P_{N_1}f_{>N_1(2)}(u_1)\|_{L^1(I_{j,N_1};L^2)}
    \|P_{N_2}u_2(\tilde{t}_j)\|_{L_x^2},
    \end{aligned}
  }
  where $V_m$ for $m=1,\dots,k+2$ is defined in \eqref{decom2} and we take $\te_1=0$ and $\te_m=1$ for $m\ge 2$.
  By \eqref{eq3.26}, it is easy to see that
  \EQQS{
    &T^{\frac{1}{4}}
     N_1^{-\frac{\al+\ga-1}{4}+1}
     \|P_{N_1}f_{>N_1(2)}(u_1)\|_{L^1(I_{j,N_1};L^2)}
     \|P_{N_2}u_2(\tilde{t}_j)\|_{L_x^2}\\
    &\le T^{\frac{1}{4}}
     N_1^{-\frac{\al+\ga-1}{4}+1}
     T N_1^{-s-\ga-1}
     \|f_{>N_1(2)}(u_1)\|_{L_T^\I H_x^s}
     \|P_{N_2}u_2\|_{L_T^\I L_x^2}\\
    &\le G(K) T^{\frac{5}{4}}
     N_1^{-\frac{\al+\ga-1}{4}-s-\ga}
     \|u_1\|_{L_T^\I H_x^s}
     \|P_{N_2}u_2\|_{L_T^\I L_x^2}.
  }
  Next, for $V_1$, by \eqref{eq_sep1}, it holds that for sufficiently small $T>0$,
  \EQQS{
    &\sum_{k=2}^{N_1(2)} \frac{|f^{(k)}(0)|}{k!}
     N_1 \|V_1\|_{L^1(I_{j,N_1};L^2)}\\
    &\les \sum_{k=2}^{N_1(2)} \frac{|f^{(k)}(0)|}{(k-1)!}
     N_1^{1-\frac{\ga}{2}}T^{\frac{1}{2}} K^{k-2}
     \|P_{N_1}u_1 P_{\le N_1(2)}u_1\|_{L^2(I_{j,N_1};L^2)}\\
    &\le G(K)(N_1^{\frac{\ga}{2}}
     \|P_{N_1}u_1\|_{L^2(I_{j,N_1};L^2)}
     +T^{\frac{1}{2}} N_1^{-s-\e}
     \|u_1\|_{L_T^\I H_x^s}).
  }
  On the other hand, for $V_m$ with $m\ge 2$, we see from \eqref{eq3.27}, \eqref{eq3.21}, and \eqref{eq3.22} that
  \EQQS{
    &\sum_{k=2}^{N_1(2)}
     \frac{|f^{(k)}(0)|}{k!} \sum_{m=2}^{k+2} T^{\frac{1}{4}}
     N_1^{-\frac{\al+\ga-1}{4}} N_1^{1-\frac{\al-1}{4}}
     \|V_m\|_{L^{\frac{4}{3}}(I_{j,N_1};L^1)}
     \|P_{N_2}u_2(\tilde{t}_j)\|_{L_x^2}\\
    &\les \sum_{k=2}^{N_1(2)}
     \frac{|f^{(k)}(0)|}{(k-2)!}K^{k-1} T N_1^{-\frac{\al+\ga-1}{4}}
     N_1^{-s-\e} \|u_1\|_{L_T^\I H_x^s}
     \|P_{N_2}u_2\|_{L_T^\I L_x^2}
     \sum_{m=2}^{k+2}\binom{k}{m-2}\\
    &\le G(2K) T N_1^{-\frac{\al+\ga-1}{4}-s-\e}
     \|u_1\|_{L_T^\I H_x^s}
     \|P_{N_2}u_2\|_{L_T^\I L_x^2}.
  }
  Collecting the above estimates, we obtain
  \EQQS{
    \sum_{j\in J_{N_1}}A_j^2
    &\les G(K)T^{\frac{1}{2}} N_1^{-\frac{\al+\ga-1}{2}}
     \|P_{N_2}u_2\|_{L_T^\I L_x^2}^2\\
    &\quad\quad\times\sum_{j\in J_{N_1}}
     (N_1^{\frac{\ga}{2}}\|P_{N_1}u_1\|_{L^2(I_{j,N_1};L^2)}
     +T^{\frac{1}{2}}N_1^{-s-\e}\|u_1\|_{L_T^\I H_x^s})^2\\
    &\le G(K)T^{\frac{1}{2}} N_1^{\frac{\ga+1-\al}{2}}
     (\|P_{N_1}u_1\|_{L_{T,x}^2}
     +N_1^{-s-\e}\|u_1\|_{L_T^\I H_x^s})^2
     \|P_{N_2}u_2\|_{L_T^\I L_x^2}^2,
  }
  which completes the estimates of $A_j$.
  Next, we consider the contribution of $B_j$.
  The contribution of $B_j$ can be estimated in a similar (and even simpler) manner, since $N_1\ge N_2$.
  When $N_2\lesssim 2^6$, as a direct consequence of \eqref{cor1}, \eqref{eq3.16} and \eqref{eq3.24} (with $h_1$ and $h_2$ switched), we obtain
  \EQS{\label{eq3.29}
    \begin{aligned}
    B_j^2
    &\les T^{\frac{1}{2}} N_1^{-\frac{\al+\ga-1}{2}}N_2^2
      \|P_{N_1}u_1(\tilde{t}_j)\|_{L_x^2}^2
      \|P_{N_2}\tilde{f}(u_2)\|_{L^1(I_{j,N_1};L_x^2)}^2\\
    &\les G(K)T^{\frac{3}{2}} N_1^{\frac{\ga+1-\al}{2}-2\gamma}N_2^{-2s}
      \|P_{N_1}u_1\|_{L^2(I_{j,N_1};L_x^2)}^2
      \|u_2\|_{L_T^\I H_x^s}^2,
    \end{aligned}
  }
  which is safely summable over small time intervals $j\in J_{N_1}$.
  On the other hand, when $N_2\gg 2^6$, we apply the decompositions \eqref{decom1} and \eqref{decom2} to $P_{N_2}\tilde{f}(u_2)$.
  Similarly to \eqref{eq3.25}, we have
  \EQQS{
    B_j
    &\les \sum_{k=2}^{N_2(2)}
     \frac{|f^{(k)}(0)|}{k!} \sum_{m=1}^{k+2} T^{\frac{1}{4}}
     N_1^{-\frac{\al+\ga-1}{4}} N_2^{1-\frac{\al-1}{4}\te_m}
     \|P_{N_1}u_1(\tilde{t}_j)\|_{L_x^2}
     \|\ti{V}_m\|_{L^{\frac{4}{4-\theta_m}}(I_{j,N_1};L^{\frac{2}{\theta_m+1}})}\\
    &\quad+T^{\frac{1}{4}}
    N_1^{-\frac{\al+\ga-1}{4}}N_2
    \|P_{N_1}u_1(\tilde{t}_j)\|_{L_x^2}
    \|P_{N_2}f_{>N_2(2)}(u_2)\|_{L^1(I_{j,N_1};L^2)},
  }
  where $\ti{V}_m$ denotes the $m$-th term in the decomposition of $P_{N_2}(u_2^k)$ given by \eqref{decom2}.
  By \eqref{eq_sep1}, \eqref{eq3.27}, \eqref{eq3.21}, \eqref{eq3.22}, and \eqref{eq3.26} with $u=u_2$, $M=N_1$, and $N=N_2$, we arrive at
  \EQQS{
    B_j
    &\les T^{\frac{1}{4}} G(2K) \bigg(\frac{N_2}{N_1}\bigg)^{\frac{\ga}{2}}
      N_1^{\frac{\ga+1-\al}{4}}N_2^{-s}
      \|P_{N_1}u_1\|_{L^2(I_{j,N_1};L^2)}
      \|u_2\|_{L_T^\infty H_x^s}\\
    &\quad+T^{\frac{1}{4}} G(K)N_1^{\frac{\ga+1-\al}{4}-\ga}N_2^{-s}
      \|P_{N_1}u_1\|_{L^2(I_{j,N_1};L^2)}
      \|u_2\|_{L_T^\infty H_x^s},
  }
  where we also used \eqref{eq3.16}. Note that here we applied Proposition \ref{prop_bistri3}  with $  M=N_1\ge N=N_2$.
  These bounds are summable over $j\in J_{N_1}$ as well.
  Finally, we consider the contribution of $D_j$.
  We see from \eqref{cor1} and Corollary \ref{cor_CK2} that
  \EQQS{
    &\bigg\|\La_a\bigg(\int_{a_{j}}^{b_{j}}
     U_\al(t-\ta) \p_x P_{\sim N_1} h_1 d\ta,
     \int_{a_{j}}^{b_{j}}
      U_\al(t-\ta) \p_x P_{\sim N_2} h_2 d\ta\bigg)\bigg\|_{L^2(I_{j,N_1};L^2)}\\
    &\lesssim T^{\frac{1}{4}}N_1^{-\frac{\al+\ga-1}{4}}
     \bigg\|\int_{a_{j}}^{b_{j}}
      U_\al(-\ta)\p_x P_{\sim N_1}h_1d\ta\bigg\|_{L_x^2}
      \bigg\|\int_{a_{j}}^{b_{j}}
       U_\al(-\ta)\p_x P_{\sim N_2}h_2d\ta\bigg\|_{L_x^2}\\
    &\lesssim T^{\frac{1}{4}}N_1^{-\frac{\al+\ga-1}{4}}
     N_1^{1-\frac{\al-1}{4}\te_1}N_2^{1-\frac{\al-1}{4}\te_2}
     \|h_1\|_{L^{\frac{4}{4-\theta_1}}(I_{j,N_1};L^{\frac{2}{\theta_1+1}})}
     \|h_2\|_{L^{\frac{4}{4-\theta_2}}(I_{j,N_1};L^{\frac{2}{\theta_2+1}})},
  }
  where $\te_m\in\{0,1\}$ for $m=1,2$.
  We then apply the same argument as in the estimates for $A_j$ and $B_j$.
  This completes the proof.
\end{proof}

\begin{proof}[Proof of Proposition \ref{prop_RefStri}]
  The proof closely follows that of Proposition \ref{prop_bistri}.
  It suffices to consider the case $N\gg 2^6$, since the case $N\les 2^6$ follows directly from the Bernstein inequality.
  As in the proof of Proposition \ref{prop_bistri}, we define small time intervals $\{I_{j,N}\}_{j\in J_N}$ such that $\bigcup_{j\in J_N} I_j=[0,T]$, $|I_j|\sim N^{-\ga}T$, and $\# J_N \sim N^\ga$.
  Since $u\in C([0,T];H^{s-}(\R))$ by Remark \ref{rem_conti}, for each $j\in J_N$ there exists $\tilde{t}_{j,N}\in I_{j,N}$ at which $\|P_N u(t)\|_{L_x^2}$ attains its minimum on $I_{j,N}$.
  In what follows, we simply write $t_j=t_{j,N}$, $\tilde{t}_j=\tilde{t}_{j,N}$, and $N(k)=N/16k$ for $k\in\N$.
  On each $I_{j,N}$, we rewrite the equation \eqref{eq2} in the Duhamel formula to obtain
  \EQQS{
    P_N u(t)
    =U_\al(t-\tilde{t}_j) P_N u(\tilde{t}_j)
     +\int_{\tilde{t}_j}^t U_\al(t-t')\p_x P_N \tilde{f}(u)dt'.
  }
  The Strichartz estimate \eqref{stri_1}, together with the H\"older inequality in time and \eqref{eq3.16} shows that
  \EQQS{
    &\|P_N u\|_{L^2(I_j;L^\I)}\\
    &\les N^{\frac{\ga+1-\al}{4}}T^{-\frac{1}{4}}
      \|P_N u\|_{L^2(I_j;L^2)}
      +T^{\frac{1}{4}}N^{-\frac{\ga}{4}}
      \left\|\int_{\tilde{t}_j}^t U_\al(t-t')\p_x P_N \tilde{f}(u)dt'\right\|_{L^4(I_j;L^\I)}.
  }
  By the Strichartz estimate \eqref{stri_2}, we have, for any function $g\in L^{\frac{4}{4-\te}}(I_j;L^{\frac{2}{1+\te}})$,
  \EQQS{
    \left\|\int_{t_{j,\min}}^{t_{j,\max}}
      U_\al(t-t')P_{\sim N} g(t',x)dt'\right\|_{L^4(I_j;L^\I)}
    \le CN^{-\frac{\al-1}{4}\te}
      \|g\|_{L^{\frac{4}{4-\te}}(I_j;L^{\frac{2}{1+\te}})},
  }
  where $\te=0,1$, $t_{j,\min}=\min I_{j,N}$, and $t_{j,\max}=\max I_{j,N}$.
  Now, we take the decompositions \eqref{decom1} and \eqref{decom2} into account, and apply the Christ-Kiselev lemma (Lemma \ref{lem_ck}) to get
  \EQQS{
    &T^{\frac{1}{4}}N^{-\frac{\ga}{4}}
     \left\|\int_{\tilde{t}_j}^t U_\al(t-t')\p_x P_N \tilde{f}(u)dt'\right\|_{L^4(I_j;L^\I)}\\
    &\les \sum_{k=2}^{N(2)}\frac{|f^{(k}(0)|}{k!}
     \sum_{l=1}^{k+2} T^{\frac{1}{4}}
     N^{-\frac{\al-1}{4}(\te_l+1)+1-\frac{\ga}{4}}
     \|V_l\|_{L^{\frac{4}{4-\te_l}}(I_j;L^{\frac{2}{1+\te_l}})}\\
    &\quad+T^{\frac{1}{4}}N^{-\frac{\al-1}{4}+1-\frac{\ga}{4}}
     \|P_{\sim N}f_{>N(2)}(u)\|_{L^{1}(I_j;L^{2})},
  }
  where $\te_1=0$, and $\te_l=1$ for $l\ge 2$.
  For $V_1$, by taking $T>0$ sufficiently small so that $G(K)T^{\frac{1}{2}}\ll 1$, we see from Proposition \ref{prop_bistri3} that
  \EQQS{
    &T^{\frac{1}{4}}
    N^{-\frac{\al-1}{4}+1-\frac{\ga}{4}}
     \|V_1\|_{L^{1}(I_j;L^{2})}\\
    &\les k K^{k-2} T^{\frac{3}{4}}
    N^{-\frac{\al-1}{4}+1-\frac{3\ga}{4}}
     \|P_N u P_{\le N(2)} u\|_{L^{2}(I_j;L^{2})}\\
    &\les k K^{k-2}G(K) T^{\frac{3}{4}} N^{\frac{\ga+1-\al}{4}+1-\frac{\al}{2}-\frac{\ga}{2}}
    (T^{-\frac{1}{2}}\|P_N u\|_{L^2(I_j;L^2)}
    +N^{-\frac{\ga}{2}-s-\e}\|u\|_{L_T^\I H_x^s})\\
    &\les k K^{k-2}G(K) T^{\frac{3}{4}}
     N^{\frac{\ga+1-\al}{4}}
    (T^{-\frac{1}{2}}\|P_N u\|_{L^2(I_j;L^2)}
    +N^{-\frac{\ga}{2}-s-\e}\|u\|_{L_T^\I H_x^s}),
  }
  since $\ga\ge 2-\al$.
  By \eqref{eq3.27}, \eqref{eq3.21}, \eqref{eq3.22} and \eqref{eq3.26}, we can close the estimate for $V_l$ with $l\ge 2$ and $f_{>N(2)}(u)$, and obtain
  \EQQS{
    &T^{\frac{1}{4}}N^{-\frac{\ga}{4}}
     \left\|\int_{\tilde{t}_j}^t U_\al(t-t')\p_x P_N \tilde{f}(u)dt'\right\|_{L^4(I_j;L^\I)}\\
    &\le CG(K)^2 T^{\frac{3}{4}}
      N^{\frac{\ga+1-\al}{4}}
      \left(T^{-\frac{1}{2}}\|P_N u\|_{L^2(I_j; L^2)}
       +N^{-\frac{\ga}{2}-s-\e}\|u\|_{L_T^\I H_x^s}\right).
  }
  Finally, collecting the above estimates, we have
  \EQQS{
    \| P_{N} u\|_{L_T^2 L_x^\I}
    &=\left[\sum_{j\in J_N} \|P_N u\|_{L^2(I_j;L^\I)}^2
      \right]^{\frac{1}{2}}\\
    &\le CT^{-\frac{1}{4}}N^{\frac{\ga+1-\al}{4}}
      \|P_N u\|_{L_{T,x}^2}
      +CG(K)^2 T^{\frac{1}{4}}
      N^{\frac{\ga+1-\al}{4}}\|P_N u\|_{L_{T,x}^2}\\
    &\quad+CG(K)^2 T^{\frac{3}{4}}
      N^{\frac{\ga+1-\al}{4}-s-\e}\|u\|_{L_T^\I H_x^s}.
  }
  Rewriting the implicit constant simply as $G(K)$, we conclude the proof.
\end{proof}

\begin{proof}[Proof of Proposition \ref{prop_bistri4}]
  The proof is identical to that of Proposition  \ref{prop_bistri}, with \eqref{cor2} replaced by \eqref{cor1}.
\end{proof}

\section{A priori estimate}\label{sec_apri}

In this section, we state the a priori estimate for the solution of \eqref{eq1}--\eqref{initial}.

\begin{prop}\label{prop_apri1}
  Let $\{\om_N^{(\de)}\}$ be an acceptable frequency weight.
  Let $0<T<1$, $\al\in[1,2]$, $s(\al)\le s\le 2$ with $s>1/2$.
  Let $s\ge s_0:=s(\al)\vee(1/2+)$.
  Let $u\in L_T^\I H_\om^s$ be a solution to \eqref{eq1} emanating from $u_0\in H_\om^s (\R)$ on $[0,T]$ and satisfies $\|u\|_{L_T^\I H_x^{s_0}}\le K$ for some $K\ge 1$.
  Then, for sufficiently small $T>0$ satisfying $G(K)T^{\frac{1}{2}}\ll 1$, there exist constants $C=C(\|u\|_{L_T^\I H_x^{s_0}})>0$ and $\te>0$
  such that
  \EQS{\label{eq_apri}
    \|u\|_{L_T^\I H_\om^s}^2
    \le \|u_0\|_{H_\om^s}^2
     +CT^\theta \|u\|_{Z_{\om,T}^s}\|u\|_{L_T^\I H_\om^s},
  }
  where $G:\R_{\ge 0}\to\R_{\ge 0}$ is a smooth increasing function.
\end{prop}

\begin{rem}
  Recall that $s(\al)$ is defined in \eqref{def_salpha}.
\end{rem}

\begin{proof}
  The strategy of the proof is the same as that of Proposition 4.6 in \cite{MT3}.
  On the other hand, we will use better estimates \eqref{eq_bistri1} and \eqref{eq_bistri2} than those were used in \cite{MT3}.
  The key point is how to recover the derivative loss.
  We recall that the exact same proof as in \cite{MT3} for the nonresonant interaction works in our case.
  Therefore, we mainly concentrate on the resonant interactions.
  Let $N\in 2^\N \cup\{0\}$ and $\e:=\min\{\frac{1}{4},\frac{1}{10}(s_0-\frac{1}{2})\}$.
  By the equation, we have
  \EQQS{
    \frac{d}{dt}\|P_N u(t)\|_{L_x^2}^2
    =-2\int_\R \p_x P_N \tilde{f}(u) P_N u dx
  }
  for $t\in [0,T]$, where $\tilde{f}(x)=f(x)-f(0)-f'(0)x$.
  Integrating in time from $0$ to $t$, and taking the $H_\om^s$-norm, we obtain
  \EQS{\label{eq4.1}
    \begin{aligned}
      &\|u(t)\|_{H_\om^s}^2\\
      &=\|u_0\|_{H_\om^s}^2
       +2 \sum_{N\in 2^\N\cup\{0\}}\sum_{k=2}^\I \om_N^2  (1\vee N)^{2s}
        \frac{f^{(k)}(0)}{k!}
       \iint_{[0,t]\times\R} u^k \p_x P_N^2 u dx dt'.
    \end{aligned}
  }
  Here, we used Hypothesis \ref{hyp2}.
  The summation over $N\les 1$ can be closed easily.
  Indeed, by a product estimate, it holds that
  \EQQS{
    \bigg|\sum_{N\les 1}\om_N^2  (1\vee N)^{2s}
    \iint_{[0,t]\times\R} \p_x P_N \tilde{f}(u) P_N u dx dt'\bigg|
    &\les \|\tilde{f}(u)\|_{L_T^2 H_\om^s}\|u\|_{L_T^2 H_\om^s}\\
    &\le TG(K)\|u\|_{L_T^\I H_\om^s}^2.
  }
  In what follows, we consider the summation over $N\gg 1$ in \eqref{eq4.1}.
  As usual, first we use a symmetrization argument, which allows us to apply integration by parts.
  By the Plancherel theorem and the Littlewood-Paley decompositioin, we obtain
  \EQQS{
    &\int_\R u^{k}\p_x P_N^2 u dx\\
    &=\frac{1}{k+1}\sum_{N_1,\dots,N_{k+1}}
     \int_{\xi_1+\cdots+\xi_{k+1}=0}
     \bigg(i\sum_{j=1}^{k+1} \phi_{N}^2(\xi_j)\xi_j\bigg)
     \prod_{j=1}^{k+1}\phi_{N_j}(\xi_j)\hat{u}(\xi_j)d\xi_1\cdots d\xi_{k+1}.
  }
  By symmetry, we may assume the following ordering of frequencies: $N_1\ge N_2\ge N_3$ if $k=2$, $N_1\ge N_2\ge N_3\ge N_4$ if $k=3$, and $N_1\ge N_2\ge N_3\ge N_4\ge \max_{j\ge 5}N_j$ if $k\ge 4$.
  We denote this condition by $(*)$ for simplicity.
  Notice that the cost of imposing this reduction is at most $(k+1)^4$.
  Moreover, it holds that $N_2\gts N_1/k$.
  See also the proof of Proposition 4.6 in \cite{MT22} for details.
  Under this observation, we are reduced to estimating the followings:
  \EQQS{
    I_{k,0}:
    &=\sum_{N\gg 1}
     \sum_{N_1,\dots, N_{k+1}:(*)}\om_N^2 N^{2s}
     \iint_{[0,t]\times\R} \Pi(P_{N_1}u,P_{N_2}u) \prod_{j=3}^{k+1} P_{N_j} u dx dt',\\
    I_{k,l}:
    &=\sum_{N\gg 1}
     \sum_{N_1,\dots, N_{k+1}:(*)}\om_N^2 N^{2s}
     \iint_{[0,t]\times\R} \p_x P_N^2 P_{N_l} u \prod_{\substack{j=1,\\j\neq l}}^{k+1} P_{N_j} u dx dt'
  }
  for $l=3,\dots,k+1$, where $\Pi(\cdot,\cdot)$ is defined in \eqref{def_pi}.
  Below, we consider only the terms $ I_{k,0}$ that are the most difficult to estimate.
  When $N_3\les 1$, Lemma \ref{lem_comm1} and the Sobolev embedding imply that
  \EQS{\label{eq4.2}
    \begin{aligned}
    |I_{k,0}|
    &\les K^{k-1} \sum_{N_1\sim N_2\gtrsim 1\gts N_3}
      \om_N^2 N_1^{2s}\LR{N_3}
      \|P_{N_1} u\|_{L_{T,x}^2}\|P_{N_2} u\|_{L_{T,x}^2}\\
    &\les T K^{k-1}\|u\|_{L_T^\I H_\om^s}^2.
    \end{aligned}
  }
  Therefore, in what follows, we always assume $N_3\gg 1$.
  It is well-known that the case $k=2$ in both terms $I_{2,0}$ and $I_{2,3}$ are nonresonant, and can be estimated at the regularity $s>1/2$ by applying the Bourgain's type estimate developed in \cite{MV15}.
  So, it suffices to consider $k\ge 3$.
  First, we estimate $I_{k,0}$ by a case by case analysis.
  More precisely, we consider three cases (1) $N_3\gg kN_4$, (2) $N_2\les N_3\les kN_4$ and (3) $N_2\gg N_3$ and $N_3\les kN_4$.
  Notice that we always have $N/4\le N_1\les kN_2$ due to  impossible interactions.
  When $N_3\gg kN_4$, the nonlinear interaction is nonresonant, and this case has already been treated in  previous studies.
  Therefore, we omit the details here.
  For a complete argument, see Case 2 in the proof of Proposition 4.8 in \cite{MT22}.
  Note that for $ \alpha>1 $ we have a gain of a negative power of the highest frequency $N_1 $ when estimating this term.
  In the two other cases, we have to apply the bilinear refined estimate for non-separated frequency functions and thus we will take advantage of Proposition \ref{prop_bistri}.
  In Case (2), where $N_2\les N_3\les kN_4$, we notice that $N_3\gts N_1/k$ since $N_2\gts N_1/k$.
  We apply Proposition \ref{prop_bistri} to the terms $\|\Pi(P_{N_1}u,P_{N_2}u)\|_{L_{T,x}^2}$ and $\|P_{N_3}u P_{N_4}u\|_{L_{T,x}^2}$, and then use the Young inequality to get
  for sufficiently small $T>0$,
  \EQQS{
    I_{k,0}
    &\les K^{k-3} \sum_{N\gg1}
      \sum_{\substack{N_1\ge N_2\ge N_3\ge N_4,\\  N_1\les kN_2\les kN_3\les k^2 N_4}}
      \om_N^2 N^{2s}
      \|\Pi(P_{N_1}u,P_{N_2}u)\|_{L_{T,x}^2}
      \|P_{N_3}uP_{N_4}u\|_{L_{T,x}^2}\\
    &\les T^{\frac{1}{2}} K^{k-2} G(K)^2
      \|u\|_{L_T^\I H_\om^s}
      \sum_{\substack{N_1\ge N_2\ge N_3\ge N_4,\\  N_1\les kN_2\les kN_3\les k^2 N_4, N_3\gg1}}
      \om_{N_1}^2\om_{N_2}^{-1} N_1^{2s+1+\frac{\ga+1-\al}{4}}N_2^{-s}\\
    &\quad\quad\quad\quad\times(T^{-\frac{1}{2}}\|P_{N_1} u\|_{L_{T,x}^2}
       + \om_{N_1}^{-1} N_1^{-s-\e}\|u\|_{L_T^\I H_\om^s})\\
    &\quad\quad\quad\quad\quad\quad\times
     N_3^{\frac{\ga+1-\al}{4}}N_4^{-s_0}
     (T^{-\frac{1}{2}}\|P_{N_3} u\|_{L_{T,x}^2}
       + N_3^{-s_0-\e}\|u\|_{L_T^\I H_x^{s_0}})\\
    &\les T^{\frac{1}{2}} k^6 2^k K^{k-2} G(K)^2
      \|u\|_{L_T^\I H_\om^s}
      \sum_{1\ll N_1\les kN_3}
      \frac{N_1}{kN_3}\\
    &\quad\quad\quad\quad\times(T^{-\frac{1}{2}}\om_{N_1}N_1^s
      \|P_{N_1} u\|_{L_{T,x}^2}
       +N_1^{-\e}\|u\|_{L_T^\I H_\om^s})\\
    &\quad\quad\quad\quad\quad\quad\quad\quad\times(T^{-\frac{1}{2}}N_3^{s_0}\|P_{N_3} u\|_{L_{T,x}^2}
       + N_3^{-\e}\|u\|_{L_T^\I H_x^{s_0}})\\
    &\les T^{\frac{1}{2}} k^6 2^k K^{k-2} G(K)^2
    \|u\|_{L_T^\I H_\om^s}^2,
  }
  since $s_0=s(\al)\vee(1/2+)$ and $k^s\le k^2$.
  Here, we used the properties of an acceptable frequency weight $\om_N^{(\de)}$ (see Definition \ref{def_AFW}).
  In particular, we used the condition $1<\de \le 2$ to obtain $\om_{N_2}^{-1}\les \de^k \om_{N_1}^{-1}\les 2^k \om_{N_1}^{-1}$ as $N_1\les kN_2$.

  Finally, we consider Case (3) where $N_2\gg N_3$ and $N_3\les kN_4$.
  In this configuration, we would like to employ both integration by parts (in the form of Lemma \ref{lem_comm1}) and the bilinear estimate \eqref{eq_bistri2} simultaneously, thereby avoiding the use of Proposition \ref{prop_bistri}.
  However, we do not know whether this argument can be rigorously justified, and in fact, we will also make use of Proposition \ref{prop_bistri} in that case.
  Following the approach taken in Case 3 in the proof of Proposition 4.6 in \cite{MT3}, we decompose the symbol $\phi_N^2(\xi_1)\xi_1+\phi_N^2(\xi_2)\xi_2$ into two parts.
  The first part behaves as a commutator term and allows us to apply \eqref{eq_bistri2}, which closes the estimate.
  The second part does not allow us to apply \eqref{eq_bistri2}, but \eqref{eq_bistri1} is still applicable.
  If this term behaved as a simple commutator estimate, we could not close the estimates for  $ \alpha \le 3/2 $, since \eqref{eq_bistri1} has some loss of regularity.
  The crucial point is that this term  actually behaves as a double commutator estimate, which enables us to close the estimates even for $ \alpha\in [1,3/2] $.

  More precisely, by decomposing the symbol $\phi_N^2(\xi_1)\xi_1+\phi_N^2(\xi_2)\xi_2$ using a Taylor expansion, we infer that for any $(\xi_1,\xi_2)\in \R^2 $ with $ |\xi_1|\sim |\xi_2|\sim N $,
  there exists $\psi(\xi_1,\xi_2)\in\R$ such that $|\psi(\xi_1,\xi_2)|\sim N$ and
  \EQS{\label{defai}
    \phi_N^2(\xi_1)\xi_1+\phi_N^2(\xi_2)\xi_2
    =a_1(\xi_2)(\xi_1+\xi_2)+
     a_2(\xi_1,\xi_2)\frac{(\xi_1+\xi_2)^2}{N},
  }
  where
  \EQS{\label{def_a1}
    \begin{aligned}
      a_1(\xi_2)
      &= \phi_N^2(\xi_2)+2\phi_N(\xi_2)\phi'(\xi_2/N)\frac{\xi_2}{N},\\
      a_2(\xi_1,\xi_2)
      &= 2\phi_N(\psi(\xi_1,\xi_2))\phi'(\psi(\xi_1,\xi_2)/N)
       + (\phi'(\psi(\xi_1,\xi_2)/N))^2 \frac{\psi(\xi_1,\xi_2)}{N}\\
      &\quad\quad+ \phi_N(\psi(\xi_1,\xi_2)) \phi''(\psi(\xi_1,\xi_2)/N)
       \frac{\psi(\xi_1,\xi_2)}{N}.
    \end{aligned}
  }
  Note that the above identities imply the estimate
  $\|a_1\|_{L^\infty(\R)}+\|a_2\|_{L^\infty(\R^2)}\les 1$.
  Moreover, we observe that
  \EQQS{
    &\int_{\xi_1+\cdots+\xi_{k+1}=0}
     a_1(\xi_2)(\xi_1+\xi_2)
     \prod_{j=1}^{k+1}\phi_{N_j}(\xi_j)\hat{u}(\xi_j)d\xi_1\cdots d\xi_{k+1}\\
    &=-\int_{\xi_1+\cdots+\xi_{k+1}=0}
     a_1(\xi_2)\bigg(\sum_{l=3}^{k+1}\xi_l\bigg)
     \prod_{j=1}^{k+1}\phi_{N_j}(\xi_j)\hat{u}(\xi_j)d\xi_1\cdots d\xi_{k+1}.
  }
  Notice that $a_1$ is a function of one variable, which allows us to apply the bilinear estimate \eqref{eq_bistri2} to $P_{N_1}u \p_x P_{N_j}u$ for $j\ge 3$ and $P_{N_2}(\check{a}_1*u) P_{N_4}u$, where $\check{a}_1$ is the Fourier inversion of $a_1$.
  On the other hand, we gain a factor $\frac{\xi_1+\xi_2}{N}$ for the contribution of $a_2$, and we make use of the bilinear estimate \eqref{eq_bistri1} to $\La_{a_2}(P_{N_1}u,P_{N_2}u)$ and $P_{N_3}uP_{N_4}u$.
  For the term involved with $a_1$, we see from the H\"older inequality, the Sobolev embedding, \eqref{eq_bistri2} and \eqref{eq2.2} that
  \EQQS{
    &\sum_{\substack{N_1\sim N_2\gg N_3\ge N_4,\\ N_1,N_3\gg 1, N_3\les kN_4 }}
      \om_{N_1}^2 N_1^{2s}
      \bigg|\iint_{[0,t]\times \R}
      P_{N_1}u P_{N_2}(\check{a}_1 * u) \p_x \bigg(\prod_{j=3}^{k+1}P_{N_j}u\bigg) dxdt'\bigg|\\
    &\les k K^{k-3} \sum_{\substack{N_1\sim N_2\gg N_3\ge N_4,\\ N_1,N_3\gg 1, N_3\les kN_4 }}
     \om_{N_1}^2 N_1^{2s}N_3\|P_{N_1}u P_{N_3}u\|_{L_{T,x}^2}
     \|P_{N_2}(\check{a}_2*u) P_{N_4}u\|_{L_{T,x}^2}\\
    &\lesssim k K^{k-3} T^{\theta-1}
     \sum_{\substack{N_1\sim N_2\gg N_3\ge N_4,\\ N_1,N_3\gg 1, N_3\les kN_4 }}
     \om_{N_1}^2 N_1^{2s-(1-\theta)(\al-1)}N_3^{1+\theta}
     U_{N_1}^{0,2}U_{N_2}^{0,2}U_{N_3}^{0,\I}U_{N_4}^{0,\I}\\
    &\lesssim k^2 K^{k-1} T^\theta \|u\|_{L_T^\I H_\om^s}^2,
  }
  where $\te\in(0,1)$ satisfies $s_0>1/2+2\te$, and
  \EQS{\label{def_U}
    U_{M}^{s,p}
    :=\|P_M u\|_{L_T^p H^s}
      +\|P_M f(u)\|_{L_T^p H^s}
  }
  for $s\ge 0$, $2\le p\le \I$.
  Next, we consider the contribution of the term involved with $a_2$ in \eqref{defai}.
  Noticing that $|\xi_1+\xi_2|=|\xi_3+\cdots+\xi_{k+1}|\les kN_3$,
  the H\"older inequality, the Sobolev embedding, \eqref{eq_bistri1}, and Remark \ref{rem_sep1} show that this contribution can be bounded by
  \EQQS{
    &\sum_{\substack{N_1\sim N_2\gg N_3\ge N_4\\ N_1,N_3\gg 1, N_3\les k N_4}}
    \sum_{M\les kN_3}
      \om_{N_1}^2 N_1^{2s-1}
      \bigg|\iint_{[0,t]\times \R}\p_x^2 P_M \La_{a_2}(P_{N_1}u,P_{N_2}u)
      \prod_{j=3}^{k+1}P_{N_j}udxdt'\bigg|\\
    &\les K^{k-3}
    \sum_{\substack{N_1\sim N_2\gg N_3\ge N_4\\ N_1,N_3\gg 1, N_3\les k N_4}}
    \sum_{M\les kN_3}
      \om_{N_1}^2 N_1^{2s-1} M^2
      \|\La_{a_2}(P_{N_1}u,P_{N_2}u)\|_{L_{T,x}^2}
      \|P_{N_3}uP_{N_4}u\|_{L_{T,x}^2}\\
    &\les k K^{k-2} T^{\frac{1}{2}} G(K) \|u\|_{L_T^\I H_\om^s}
     \sum_{N_1\gg N_3}
     \om_{N_1} N_1^{s-1}N_3^2 N_1^{\frac{\ga+1-\al}{4}}
     N_3^{\frac{\ga+1-\al}{4}-s_0}\\
    &\quad\quad\quad\quad\times
     \Big(\|P_{N_1} u\|_{L_{T,x}^2}
     +\om_{N_1}^{-1} N_1^{-s-\e}\|u\|_{L_T^\I H_\om^s}\Big)
     \Big(\|P_{N_3} u\|_{L_{T,x}^2}
      +N_3^{-s_0-\e}\|u\|_{L_T^\I H_x^{s_0}}\Big)\\
    &\les k K^{k-2} T^{\frac{1}{2}} G(K) \|u\|_{L_T^\I H_\om^s}
     \sum_{N_1\gg N_3}
     \bigg(\frac{N_3}{N_1}\bigg)^{\frac{3}{4}}
     \om_{N_1} N_1^s N_3^{s_0}\\
    &\quad\quad\quad\quad\times
     \Big(\|P_{N_1} u\|_{L_{T,x}^2}
     +\om_{N_1}^{-1} N_1^{-s-\e}\|u\|_{L_T^\I H_\om^s}\Big)
     \Big(\|P_{N_3} u\|_{L_{T,x}^2}
      +N_3^{-s_0-\e}\|u\|_{L_T^\I H_x^{s_0}}\Big)\\
    &\le k K^{k-1} T^{\frac{1}{2}} G(K) \|u\|_{L_T^\I H_\om^s}^2,
  }
  where we used that $2s_0\ge 1+\frac{\gamma+1-\alpha}{2}$, $T>0$ is sufficiently small, and $\e>0$ satisfies $s_0>1/2+2\te+\e$.
  In conclusion, collecting the above estimates and using Hypothesis \ref{hyp2}, we have \eqref{eq_apri}.
  For more details, see the proof of Proposition 4.6 in \cite{MT3}.
  One only has to replace the the bilinear estimates in \cite{MT3} by the bilinear estimates \eqref{eq_bistri1} and \eqref{eq_bistri2}.
  This completes the proof.
\end{proof}

\begin{rem}
For $ \alpha\in ]\frac32,2] $, we can simplify the treatment of Case (3) in the above proof.
Indeed, since \eqref{eq_bistri1} does not exhibit any loss of regularity, we can directly apply \eqref{eq_bistri1} to both $ \|\Pi(P_{N_1}u, P_{N_2} u)\|_{L_{T,x}^2}$ and $\|P_{N_3}u P_{N_4}u\|_{L_{T,x}^2}$.
More precisely, it suffices to consider only the first-order term in the Taylor expansion, unlike in \eqref{defai}:
\EQQS{
  \phi_N^2(\xi_1)\xi_1+\phi_N^2(\xi_2)\xi_2
  =\left(\phi_N(\nu)^2
    +2\phi_N(\nu)
     \phi'\left(\frac{\nu}{N}\right)\frac{\nu}{N}\right)
   (\xi_1+\xi_2),
}
for some $\nu\in (|\xi_1|,|\xi_2|)$ or $\nu\in (|\xi_2|,|\xi_1|)$.
This reflects that $\Pi(P_{N_1}u, P_{N_2} u) $ behaves as a commutator term.
\end{rem}

\section{Estimate for the difference}
\label{sec_diff}

In this section, we discuss estimates for the difference of two solutions.
To this end, we prove several key estimates in the spirit of Proposition \ref{prop_bistri}.
However, when dealing with the difference, the equation exhibits less symmetry than in the case of a single solution, which requires a more delicate analysis.
We begin by stating the bilinear estimate for non-separated frequency functions, which corresponds to Proposition \ref{prop_bistri}.

\begin{prop}\label{prop_bistri_d1}
  Let $0<T<1$, $\al\in[1,2]$, $\ga=\max\{2-\al,1-\al/3\}$, $N_1\ge N_2\ge 1$, $s(\al)\le s\le 2$ with $s>1/2$, and $0<\e\ll 1$ so that $s>1/2+\e$.
  Let $a\in L^\I(\R^2)$ be such that $\|a\|_{L^\I}\lesssim 1$.
  Finally, let $u_1,u_2\in L^\I(]0,T[;H^s(\R))$ satisfy \eqref{eq1} on $]0,T[\times \R$ and
  \EQQS{
    \|u_1\|_{L_T^\I H_x^s}+\|u_2\|_{L_T^\I H_x^s}\le K
  }
  for some $K\ge 1$.
  Then, there exists a smooth increasing function $G:\R_{\ge 0}\to \R_{\ge 0}$ such that for sufficiently small $T>0$ satisfying $G(K)T^{\frac{1}{2}}\ll 1$, it holds that
  \EQS{\label{eq_bistri_d1}
    \begin{aligned}
      &\|\La_a(P_{N_1}w,P_{N_2}w)\|_{L_{T,x}^2}\\
      &\les T^{\frac{1}{4}} G(K) N_1^{\frac{\ga+1-\al}{4}}
      N_2^{1-s}
      \Big(T^{-\frac{1}{2}}\|P_{N_1}w\|_{L_{T,x}^2}
      +N_1^{1-s-\e}\|w\|_{L_T^\I H_x^{s-1}}\Big)
      \|w\|_{L_T^\I H_x^{s-1}},
    \end{aligned}
  }
  where $w=u_1-u_2$.
\end{prop}

As in Section \ref{sec_boot}, the following estimate on a short time interval is the key ingredient in the proof of Proposition \ref{prop_bistri_d1}.

\begin{prop}\label{prop_sep2}
  Let $0<T<1$, $\al\in[1,2]$, $\ga=\max\{2-\al,1-\al/3\}$, $M\ge N\ge 1$, $s>1/2$, and $0<\e\ll 1$ so that $s>1/2+\e$.
  Let $I\subset \R$ be a closed interval of length $|I|\sim M^{-\ga}T$.
  Assume that $u_1,u_2\in L^\I(]0,T[;H^s(\R))$ satisfies \eqref{eq1} on $]0,T[\times \R$ and $\|u_1\|_{L_T^\I H_x^s}+\|u_2\|_{L_T^\I H_x^s}\le K$ for some $K\ge 1$.
  Then, there exists a smooth increasing function $G:\R_{\ge 0}\to \R_{\ge 0}$ such that for sufficiently small $T>0$ satisfying $G(K)T^{\frac{1}{2}}\ll 1$, it holds that
  \EQS{\label{eq_sep2}
    \begin{aligned}
      &\|P_{N}w P_{\le 2^{-5}N}z\|_{L^2(I;L^2)}\\
      &\les N^{-\frac{\al}{2}}G(K)
       \Big(T^{-\frac{1}{2}}M^{\frac{\ga}{2}}
        \|P_{N}w\|_{L^2(I;L^2)}
        +N^{1-s-\e}\|w\|_{L_T^\I H_x^{s-1}}\Big),
    \end{aligned}
  }
  where $w=u_1-u_2$ and $z=u_1$ or $u_2$.
\end{prop}

\begin{proof}
  The proof proceeds essentially in parallel with that of Proposition \ref{prop_bistri3}.
  However, one needs to be more careful since the equation satisfied by $w$ has less symmetry.
  It suffices to consider the case $N\ge 2^6$.
  By \eqref{eq2}, $w=u_1-u_2$ satisfies
  \EQS{\label{eq3}
    \p_t w+(L_{\al+1}+f'(0)\p_x)w
     +\p_x \de(u_1,u_2)=0,
  }
  where $\de (u_1,u_2):=\tilde{f}(u_1)-\tilde{f}(u_2)$.
  For simplicity, we also put $N(k):=N/(16k)$.
  We may set $I=[t_1,t_2]\subset\R$ with $t_1<t_2$.
  Let $t_0\in I$ be the point at which $\|P_N w\|_{L_x^2}$ attains its minimum over $I$, so that \eqref{eq3.16} holds.
  By applying the Duhamel principle to $w$ and $z$, we observe that
  \EQQS{
    &\|P_N w P_{\le N(2)} z \|_{L^2(I;L^2)}\\
    &\le \|U_\al(t-t_0) P_N w(t_0) U_\al(t-t_0) P_{\le N(2)} z(t_0) \|_{L^2(I;L^2)}\\
    &\quad+\bigg\| U_\al(t-t_0) P_N w(t_0) \bigg(\int_{t_0}^t U_\al(t-t')\p_x P_{\le N(2)} \tilde{f}(z) dt'\bigg)  \bigg\|_{L^2(I;L^2)}\\
    &\quad+\bigg\| \bigg(\int_{t_0}^t U_\al(t-t')\p_x P_N \delta(u_1,u_2) dt'\bigg)
     U_\al(t-t_0) P_{\le N(2)} z(t_0) \bigg\|_{L^2(I;L^2)}\\
    &\quad+\bigg\| \bigg(\int_{t_0}^t U_\al(t-t')\p_x P_N
    \delta(u_1,u_2) dt'\bigg)
     \int_{t_0}^t U_\al(t-t')\p_x P_{\le N(2)} \tilde{f}(z) dt'  \bigg\|_{L^2(I;L^2)}\\
    &=:A_1+A_2+A_3+A_4.
  }
  Similarly to the proof of Proposition \ref{prop_bistri3}, we see from \eqref{stri_2}, \eqref{cor2.2} and the Christ-Kiselev lemma (Corollary \ref{cor_CK1}) that
  \EQQS{
    A_1+A_2
    \les N^{-\frac{\al}{2}}
     T^{-\frac{1}{2}}M^{\frac{\ga}{2}}
     (1+G(K))\|P_N w\|_{L^2(I;L^2)}.
  }
  We only consider the estimate for $A_3$, since $A_4$ can be estimated in the same way as $A_3$, thanks to \eqref{eq3.17}.
  For that purpose, we need to take the decompositions \eqref{decom1} and \eqref{decom2} into account.
  First, note that $u_1^k-u_2^k=w(u_1^{k-1}+u_1^{k-2}u_2+\cdots +u_1u_2^{k-2}+u_2^{k-1})$.
  By a slight abuse of notation, we write either $u_1$ or $u_2$ as $z$ for simplicity, so that $u_1^k-u_2^k=kwz^{k-1}$, where $w=u_1-u_2$.
  Moreover, we decompose $P_N (wz^{k-1})$ as follows:
  \EQS{\label{decom3}
    \begin{aligned}
      P_N (wz^{k-1})
      &=P_{\sim N}(P_{N}w
       P_{\le N(2)}z (P_{\le N(k)}z)^{k-2})\\
      &\quad+P_{\sim N}(P_{N}w
       P_{[N(2),N(k)]}z (P_{\le N(k)}z)^{k-2})\\
      &\quad+P_{\sim N}
        ([P_{N}, (P_{\le N(k)}z)^{k-1}]
        P_{\sim N}w)\\
      &\quad +\sum_{j=1}^{k-1}\binom{k-1}{j}
      P_N(P_{\le N(k)}w
      (P_{> N(k)}z)^j
      (P_{\le N(k)}z)^{k-1-j})\\
      &\quad+\sum_{j=1}^{k-1}\binom{k-1}{j}
      P_N(P_{> N(k)}w
      (P_{> N(k)}z)^j
      (P_{\le N(k)}z)^{k-1-j})\\
      &=:P_{\sim N}\bigg(V_1+V_2+V_3+\sum_{j=1}^{k-1}V_{j+3}
       +\sum_{j=1}^{k-1}V_{j+k+2}\bigg).
    \end{aligned}
  }
  Recall that $P_{[N(2),N(k)]}z=P_{\le N(2)}z-P_{\le N(k)}z$.
  With this notation, we have
  \EQQS{
    A_3
    &\le \sum_{k=2}^{N(2)}\sum_{j=1}^{2k+1} \frac{|f^{(k)}(0)|}{(k-1)!}
     \bigg\| \bigg(\int_{t_0}^t U_\al(t-t')\p_xP_{\sim N}
     V_j dt'\bigg)
     U_\al(t-t_0) P_{\le N(2)} z(t_0) \bigg\|_{L^2(I;L^2)}\\
    &\quad+
     \bigg\| \bigg(\int_{t_0}^t U_\al(t-t')\p_x
      P_N \de_{>N(2)}(u_1,u_2) dt'\bigg)
      U_\al(t-t_0) P_{\le N(2)} z(t_0) \bigg\|_{L^2(I;L^2)}\\
   &=:\sum_{k=2}^{N(2)}\sum_{j=1}^{2k+1}
    \frac{|f^{(k)}(0)|}{(k-1)!}A_{3,1,j}+A_{3,2},
  }
  where $\de_{>N(2)}(u_1,u_2):= \de(u_1,u_2)- \sum_{k=2}^{N(2)} \frac{f^{(k)}(0)}{(k-1)!}wz^{k-1}$.
  We now consider $A_{3,1,j}$.
  By \eqref{toto} with $\te=0$ and the Christ-Kiselev lemma (Corollary \ref{cor_CK1}), we also obtain
  \EQS{\label{eq5.7}
    \begin{aligned}
      A_{3,1,1}
      &\les N^{1-\frac{\al}{2}}
      \|P_{N}w(P_{\le N(k)}z)^{k-1}\|_{L^1(I;L^2)}
      \|P_{\le N(k)}z\|_{L_T^\I L_x^2}\\
      &\les T^{\frac{1}{2}} K^{k-1}N^{1-\frac{\al}{2}}M^{-\frac{\ga}{2}}
       \|P_{N}wP_{\le N(k)}z\|_{L^2(I;L^2)}\\
      &\les  T^{\frac{1}{2}} K^{k-1}
       \|P_{N}wP_{\le N(k)}z\|_{L^2(I;L^2)},
    \end{aligned}
  }
  since $M\ge N$.
  Similarly, as in \eqref{eq3.27}, \eqref{toto} with $\te=1$ and the Christ-Kiselev lemma (Corollary \ref{cor_CK1}) show that
  \EQS{\label{eq5.9}
    \begin{aligned}
      A_{3,1,2}
      &\les N^{-\frac{\al}{2}}
      \|P_{N}w
       P_{[N(2),N(k)]}z (P_{\le N(k)}z)^{k-2}\|_{L^{\frac{4}{3}}(I;L^1)}
      \|P_{\le N(2)}z\|_{L_T^\I L_x^2}\\
      &\les T^{\frac{3}{4}} (k-1)K^k N^{-\frac{\al}{2}} N^{1-s-\e}
       \|w\|_{L_T^\I H_x^{s-1}},
    \end{aligned}
  }
  since $3-\al-3\ga\le 0$.
  In the same way, $A_{3,1,3}$ can be bounded by the right-hand side of \eqref{eq5.9}.
  For $A_{3,1,j}$ with $j=4,\dots,k+2$, by \eqref{toto} with $\te=1$ and the Christ-Kiselev lemma (Corollary \ref{cor_CK1}), we obtain
  \EQS{\label{eq5.5}
    \begin{aligned}
      \sum_{j=1}^{k-1}A_{3,1,j+3}
      &\les \sum_{j=1}^{k-1}
       N^{-\frac{\al}{2}}N^{1-\frac{\al-1}{4}}
       \|V_{j+3}\|_{L^{\frac{4}{3}}(I;L^1)}
       \|P_{\le N(k)}z\|_{L_T^\I L_x^2}\\
      &\les 2^k K^{k-1} N^{-\frac{\al}{2}}N^{\frac{5-\al-3\ga}{4}}
       T^{\frac{3}{4}}
       \|P_{\le N(k)}w\|_{L_T^\I L_x^2}
       \|P_{> N(k)}z\|_{L_T^\I L_x^2}\\
      &\les T^{\frac{3}{4}} k^2 (2K)^k N^{-\frac{\al}{2}}
       N^{1-s-\e}
       \|w\|_{L_T^\I H_x^{s-1}}.
    \end{aligned}
  }
  Here, we used $3-\al-3\ga\le 0$ and $N(k)^{-s}\le 2^{10} k^2 N^{-s}$.
  Next, we consider $A_{3,2,j}$.
  First, we observe that
  \EQQS{
    &\|P_N(P_{> N(k)}w (P_{> N(k)}z)^j
     (P_{\le N(k)}z)^{k-1-j})\|_{L^{\frac{4}{3}}(I;L^1)}\\
    &\les T^{\frac{3}{4}}M^{-\frac{3}{4}\ga}
     \sum_{N_1> N(k)}\|P_N(P_{N_1}w (P_{> N(k)}z)^j
      (P_{\le N(k)}z)^{k-1-j})\|_{L_T^\I L_x^1}.
  }
  We divide this summation into two cases: $N_1\ge 2^3 N$ and $N(k)<N_1\le 2^2 N$.
  For the first case, by the impossible frequency interaction, we have
  \EQQS{
    &\sum_{N_1\ge 2^3 N}\|P_N(P_{N_1}w (P_{> N(k)}z)^j
     (P_{\le N(k)}z)^{k-1-j})\|_{L_T^\I L_x^1}\\
    &\les K^{k-1-j} \sum_{N_1\ge 2^3 N} \|P_{N_1} w\|_{L_T^\I L_x^2}
     \|P_{\sim N_1}((P_{> N(k)}z)^j)\|_{L_T^\I L_x^2}\\
    &\les K^{k-1-j} \|w\|_{L_T^\I H_x^{s-1}}
     \|(P_{> N(k)}z)^j\|_{L_T^\I H_x^s} \sum_{N_1> 2^3 N}N_1^{1-2s}
    \les K^{k-1} N^{1-2s} \|w\|_{L_T^\I H_x^{s-1}}.
  }
  On the other hand, for the second case, it holds that
  \EQQS{
    &\sum_{N(k)<N_1\le 2^2 N}\|P_N(P_{N_1}w (P_{> N(k)}z)^j
     (P_{\le N(k)}z)^{k-1-j})\|_{L_T^\I L_x^1}\\
    &\les K^{k-1} N(k)^{-s}
     \sum_{N(k)<N_1\le 2^2 N} \|P_{N_1}w\|_{L_T^\I L_x^2}
    \les K^{k-1}k^3N^{1-2s}\|w\|_{L_T^\I H_x^{s-1}}.
  }
  Therefore, \eqref{toto} with $\te=1$ and the Christ-Kiselev lemma (Corollary \ref{cor_CK1}) give
  \EQS{\label{eq5.10}
    \begin{aligned}
      \sum_{j=1}^{k-1}A_{3,1,j+k+2}
      &\les \sum_{j=1}^{k-1}
       N^{-\frac{\al}{2}}N^{1-\frac{\al-1}{4}}
       \|V_{j+3}\|_{L^{\frac{4}{3}}(I;L^1)}
       \|P_{\le N(2)}z\|_{L_x^2}\\
      &\les T^{\frac{3}{4}} k^3 (2K)^k
       N^{-\frac{\al}{2}} N^{\frac{3}{2}-2s}
       \|w\|_{L_T^\I H_x^{s-1}}.
    \end{aligned}
  }
  It is worth mentioning that $3/2-2s<1-s$ since $s>1/2$.
  Finally, we consider $A_{3,2}$.
  By the product estimate \eqref{eq2.3}, we have
  \EQS{\label{eq5.11}
    \begin{aligned}
      \|P_N \de_{> N(2)}(u_1,u_2)\|_{L_x^2}
      &\les N^{1-s}\sum_{l=N(2)+1}^\I
       \frac{|f^{(l)}(0)|}{(l-1)!} \|wz^{l-1}\|_{H_x^{s-1}}\\
      &\les N^{1-s}N(2)^{-1}G(K)\|w\|_{H_x^{s-1}}
      \les N^{-s}G(K)\|w\|_{H_x^{s-1}},
    \end{aligned}
  }
  which together with \eqref{toto} and the Christ-Kiselev lemma (Corollary \ref{cor_CK1}) imply that
  \EQS{\label{eq5.6}
    \begin{aligned}
      A_{3,2}
      &\les N^{1-\frac{\al}{2}}
       \|P_N \de_{> N(2)}(u_1,u_2)\|_{L^1(I;L^2)}
       \|P_{\le N(2)}z\|_{L_T^\I L_x^2}\\
      &\les T N^{-\frac{\al}{2}}N^{1-s-\ga}KG(K)
       \|w\|_{H_x^{s-1}}.
    \end{aligned}
  }
  In conclusion, we obtain
  \EQQS{
    A_3
    \les T^{\frac{1}{2}}G(K)\|P_{N}wP_{\le N(k)}z\|_{L^2(I;L^2)}
     +T^{\frac{3}{4}} G(K) N^{-\frac{\al}{2}} N^{1-s-\e}
      \|w\|_{L_T^\I H_x^{s-1}}
  }
  since $2-\al-\ga\le 0$.
  This completes the estimates for $A_3$.
  The estimates for $A_4$ follows from the above argument with \eqref{eq3.17} and the Christ-Kiselev lemma (Corollary \ref{cor_CK2}) instead of Corollary \ref{cor_CK1}.
  Collecting the above estimats, we obtain
  \EQQS{
    \|P_{N}wP_{\le N(2)}z\|_{L^2(I;L^2)}
    &\les N^{-\frac{\al}{2}}G(K)
     \Big(T^{-\frac{1}{2}}M^{\frac{\ga}{2}}
      \|P_{N}w\|_{L^2(I;L^2)}
      +N^{1-s-\e}\|w\|_{L_T^\I H_x^{s-1}}\Big)\\
    &\quad+T^{\frac{1}{2}}G(K)\|P_{N}wP_{\le N(2)}z\|_{L^2(I;L^2)}.
  }
  Recalling that $z$ denotes either $u_1$ or $u_2$, this leads to
  \EQQS{
    &\|P_{N}wP_{\le N(2)}u_1\|_{L^2(I;L^2)}
    +\|P_{N}wP_{\le N(2)}u_2\|_{L^2(I;L^2)}\\
    &\les N^{-\frac{\al}{2}}G(K)
     \Big(T^{-\frac{1}{2}}M^{\frac{\ga}{2}}
      \|P_{N}w\|_{L^2(I;L^2)}
      +N^{1-s-\e}\|w\|_{L_T^\I H_x^{s-1}}\Big)\\
    &\quad+T^{\frac{1}{2}}G(K)\Big(\|P_{N}wP_{\le N(2)}u_1\|_{L^2(I;L^2)}
     +\|P_{N}wP_{\le N(2)}u_2\|_{L^2(I;L^2)}\Big).
  }
  By choosing $T>0$ sufficiently small so that $T^{\frac{1}{2}}G(K)\ll1$, we obtain \eqref{eq_sep2}, which completes the proof.
\end{proof}

Now we are ready to show Proposition \ref{prop_bistri_d1}.

\begin{proof}[Proof of Proposition \ref{prop_bistri_d1}]
  The proof proceeds essentially in parallel with the argument of Proposition \ref{prop_bistri}.
  First, by an argument similar to that of \eqref{eq3.28}, it suffices to consider the case $N_1\gg 2^6$.
  The main difference is that we use \eqref{eq_sep2} instead of \eqref{eq_sep1}.
  We divide the time interval $[0,T]$ into closed subintervals of length $\sim N_1^{-\ga} T $; that is, we define $\{I_{j,N_1}\}_{j\in J_{N_1}}$ with $  \# J_{N_1}\sim N_1^\ga $ such that $\bigcup_{j\in J_{N_1}}I_{j,N_1}=[0,T]$ and $|I_{j,N_1}|\sim N_1^{-\ga} T$.
  Without loss of generality, we may set $I_{j,N_1}=[t_{j,N_1},t_{j+1,N_1}]$ for $j\in J_{N_1}$.
  By continuity, there exists $\tilde{t}_{j,N_1}\in I_{j,N_1}$ at which $\|P_{N_1}w(t)\|_{L_x^2}$ attains its minimum on $I_{j,N_1}$.
  For simplicity, we write $t_j=t_{j,N_1}$, $\tilde{t}_j=\tilde{t}_{j,N_1}$ and $N_m(k):=N_m/(16k)$ for $m=1,2$ and $k\in \N$.
  On each $ I_{j,N_1} $, by the Duhamel principle for \eqref{eq3}, it holds that for $m=1,2$
  \EQQS{
    P_{N_m}w(t)
    &=U_\al(t-\tilde{t}_j)P_{N_m} w(\tilde{t}_j)
     +\int_{\tilde{t}_j}^t U_\al (t-\ta)\p_x P_{N_m} \de(u_1,u_2)d\ta\\
    &=:U_\al(t-\tilde{t}_j)P_{N_m} w(\tilde{t}_j)
     +F_{m,j},
  }
  where $\de(u_1,u_2)=\tilde{f}(u_1)-\tilde{f}(u_2)$.
  Therefore
  \EQQS{
    &\|\La_a (P_{N_1} w, P_{N_2} w)\|_{L_{T,x}^2}^2\\
    &\le \sum_{j\in J_{N_1}}
     \|\La_a(U_\al(t-\tilde{t}_j)P_{N_1}w(\tilde{t}_j),
      U_\al(t-\tilde{t}_j)P_{N_2}w(\tilde{t}_j))\|_{L^2(I_{j,N_1};L^2)}^2\\
    &\quad+\sum_{j\in J_{N_1}}A_{j}^2
     +\sum_{j\in J_{N_1}} B_{j}^2
     +\sum_{j\in J_{N_1}} D_{j}^2,
  }
  where
  \EQQS{
    A_{j}&=\|\La_a(F_{1,j},
     U_\al(t-\tilde{t}_j)P_{N_2}w(\tilde{t}_j))\|_{L^2(I_{j,N_1};L^2)},\\
    B_{j}&=\|\La_a(U_\al(t-\tilde{t}_j)P_{N_1}w(\tilde{t}_j),
     F_{2,j})\|_{L^2(I_{j,N_1};L^2)},\\
    D_{j}&=\|\La_a(F_{1,j},
     F_{2,j})\|_{L^2(I_{j,N_1};L^2)}.
  }
  In the same way as in \eqref{eq3.23}, we obtain
  \EQQS{
    &\sum_{j\in J_{N_1}}
     \|\La_a(U_\al(t-\tilde{t}_j)P_{N_1}w(\tilde{t}_j),
      U_\al(t-\tilde{t}_j)P_{N_2}w(\tilde{t}_j))\|_{L^2(I_{j,N_1};L^2)}^2\\
    &\les T^{-\frac 12}N_1^{\frac{\ga+1-\al}{2}}
    \|P_{N_1} w\|_{L^2_{T,x}}^2\|P_{N_2} w\|_{L_T^\I L^2_x}^2.
  }
  Now we focus on the estimates for $A_{j}, B_{j}$ and $D_{j}$.
  We begin with $A_{j}$.
  For that purpose, we take the decomposition \eqref{decom3} into account:
  \EQQS{
    A_j
    &\le \sum_{k=2}^{N_1(2)}\sum_{l=1}^{2k+1} \frac{|f^{(k)}(0)|}{(k-1)!}
     \bigg\| \bigg(\int_{\tilde{t}_j}^t U_\al(t-t')\p_xP_{\sim N_1}
     V_l dt'\bigg)
     U_\al(t-\tilde{t}_j) P_{N_2} w(\tilde{t}_j) \bigg\|_{L^2(I_{j,N_1};L^2)}\\
    &\quad+
     \bigg\| \bigg(\int_{\tilde{t}_j}^t U_\al(t-t')\p_x
      P_{N_1} \de_{>N_1(2)}(u_1,u_2) dt'\bigg)
      U_\al(t-\tilde{t}_j) P_{N_2} w(\tilde{t}_j) \bigg\|_{L^2(I_{j,N_1};L^2)}\\
   &=:\sum_{k=2}^{N(2)}\sum_{l=1}^{2k+1}
    \frac{|f^{(k)}(0)|}{(k-1)!}A_{j,1,l}+A_{j,2},
  }
  where $V_l$ for $l=1,\dots,2k+1$ is defined in \eqref{decom3} with $N=N_1$, and $\de_{>N_1(2)}(u_1,u_2):= \de(u_1,u_2)- \sum_{k=2}^{N_1(2)} \frac{f^{(k)}(0)}{(k-1)!}wz^{k-1}$.
  Here, we denote either $u_1$ or $u_2$ by $z$.
  By a slight abuse of notation, we write $u_1^k-u_2^k=kwz^{k-1}$.
  As for $A_{j,1,1}$, \eqref{eq3.24} with $\te=0$ shows that
  for sufficiently small $T>0$ satisfying $G(K)T^{\frac{1}{2}}\ll 1$,
  \EQQS{
    A_{j,1,1}
    &\les T^{\frac{1}{4}}N_1^{1-\frac{\al+\ga-1}{4}}
     \|P_{N_1}wP_{\le N_1(2)}z
      (P_{\le N_1(k)}z)^{k-2}\|_{L^1(I_{j,N_1};L^2)}
     \|P_{N_2}w(\tilde{t}_j)\|_{L_x^2}\\
    &\les T^{\frac{3}{4}}N_1^{-\frac{\al+\ga-1}{4}}
     K^{k-2}G(K)
     (T^{-\frac{1}{2}}N_1^{\frac{\ga}{2}}
     \|P_{N_1}w\|_{L^2(I_{j,N_1};L^2)}
     +N_1^{1-s-\e}\|w\|_{L_T^\I H_x^{s-1}})\\
    &\quad\times\|P_{N_2}w\|_{L_T^\I L_x^2},
  }
  since $2-\al-\ga\le 0$.
  In order to estimate $A_{j,1,l}$ with $l\ge 2$ and $A_{j,2}$, we follow the arguments in \eqref{eq5.9}, \eqref{eq5.5}, \eqref{eq5.10} and \eqref{eq5.6}, using \eqref{cor1} instead of \eqref{cor2.2}.
  As a result, all the estimates hold with $N^{-\frac{\al}{2}}$ replaced by $T^{\frac{1}{4}}N_1^{-\frac{\al+\ga-1}{4}}$.
  Therefore, we obtain
  \EQQS{
    A_j
    &\les \sum_{k=2}^{N_1(2)}\frac{|f^{(k)}(0)|}{(k-1)!}
     T^{\frac{3}{4}} N_1^{-\frac{\al+\ga-1}{4}}K^{k-1}G(K)
     \|P_{N_2}w\|_{L_T^\I L_x^2}\\
    &\quad\quad\quad \times(N_1^{\frac{\ga}{2}}
     \|P_{N_1}w\|_{L^2(I_{j,N_1};L^2)}
     +N_1^{1-s-\e}\|w\|_{L_T^\I H_x^{s-1}})\\
    &\quad+T^{\frac{5}{4}}N_1^{-\frac{\al+\ga-1}{4}}N_1^{1-s-\ga}
     \|w\|_{L_T^\I H_x^{s-1}}\|P_{N_2}w\|_{L_T^\I L_x^2}\\
    &\les T^{\frac{1}{4}}N_1^{\frac{\ga+1-\al}{4}}G(K)
     (\|P_{N_1}w\|_{L^2(I_{j,N_1};L^2)}
     +N_1^{1-s-\frac{\ga}{2}-\e}\|w\|_{L_T^\I H_x^{s-1}})
     \|P_{N_2}w\|_{L_T^\I L_x^2},
  }
  which implies that
  \EQQS{
    \sum_{j\in J_{N_1}} A_j^2
    \les T^{\frac{1}{2}}N_1^{\frac{\ga+1-\al}{2}} G(K)^2
    (\|P_{N_1}w\|_{L_{T,x}^2}+N_1^{1-s-\e}\|w\|_{L_T^\I H_x^{s-1}})^2
    \|P_{N_2}w\|_{L_T^\I L_x^2}^2.
  }
  Next, we consider the contributions of $B_j$ and $D_j$.
  As discussed in \eqref{eq3.29}, it suffices to treat the case $N_2\gg 2^6$.
  For $B_j$, we take the decomposition \eqref{decom3} into account and apply the same argument as for $A_j$, using $N_1\ge N_2$.
  It is worth mentioning that we do not need to retain the $L^2$-norm in time for $P_{N_2}w$ when considering the term corresponding to $V_1$ in \eqref{decom3}.
  Specifically, we apply \eqref{eq_sep2} in the following form:
  \EQQS{
    \|P_{N_2}w P_{\le N_2(2)}z\|_{L^2(I_{j,N_1};L^2)}
    \les N_2^{-\frac{\al}{2}+1-s}G(K) \|w\|_{L_T^\I H_x^{s-1}}.
  }
  The remaining terms can be estimated as in the case of $A_j$, applying \eqref{eq3.16} (in fact, these terms have slightly better regularity).
  Moreover, we can similarly estimate $D_j$, using Corollary \ref{cor_CK2} as a replacement for Corollary \ref{cor_CK1}.
  Combining all the estimates, we obtain \eqref{eq_bistri_d1}, which completes the proof.
\end{proof}

\begin{prop}\label{prop_diff1}
  Let $0<T<1$, $\al\in[1,2]$ and $s(\al)\le s\le 2$ with $s>1/2$.
  Let $u_1,u_2\in L_T^\I H_x^s$ be solutions to \eqref{eq1} on $[0,T]$  emanating respectively from initial data $u_{0,1},u_{0,2}\in H^s (\R)$, and satisfying $\|u_1\|_{L_T^\I H_x^{s}}+\|u_2\|_{L_T^\I H_x^{s}}\le K$ for some $K\ge 1$.
  Then, there exist a smooth increasing function $G:\R_{\ge 0}\to \R_{\ge 0}$ and constants $C>1$ and $\te>0$ such that for sufficiently small $T>0$ satisfying $G(K)T^{\frac{1}{2}}\ll 1$, we have
  \EQQS{
    \|w\|_{L_T^\I H_x^{s-1}}^2
    \le \|u_{0,1}-u_{0,2}\|_{H_x^{s-1}}^2
     +CT^\te \|w\|_{Z_{T}^{s-1}}\|w\|_{L_T^\I H_x^{s-1}},
  }
  where $w=u_1-u_2$.
\end{prop}

\begin{proof}
  The proof relies on the argument of Proposition 5 in \cite{MT3}.
  For simplicity, we denote $u_1$ or $u_2$ by $z$, so that $u_1^k-u_2^k=kwz^{k-1}$.
  By using the equation \eqref{eq3}, a symmetrization argument, and the product estimate \eqref{eq2.3ana} for low frequencies $N\les 1$ (see \eqref{eq5.11} for instance), it follows that for $t\in[0,T]$
  \EQQS{
    \|w(t)\|_{H_x^{s-1}}^2
    \le \|u_{0,1}-u_{0,2}\|_{H_x^{s-1}}^2
     + CT\|w\|_{L_T^\I H_x^{s-1}}^2
     +\sum_{k=2}^\I \frac{|f^{(k)}(0)|}{(k-1)!}|I_{t,k}|,
  }
  where $I_{t,k}$ is defined by
  \EQQS{
    I_{t,k}
    :=C(k)\sum_{N\gg 1} \sum_{\substack{N_1\ge N_2,\\ N_3\ge N_4\ge \max_{j\ge 5} N_j}} N^{2(s-1)}
    \int_0^t \int_\R
    \Pi (P_{N_1}w,P_{N_2}w)
    \prod_{j=3}^{k+1} P_{N_j} z dxdt',
  }
  with $|C(k)|\les k^2$, which stems from the frequency ordering assumptions.
  More precisely, we assume $N_3\ge N_4\ge N_5=\max_{j\ge 5} N_j$ when $k\ge 4$, whereas we assume $N_3\ge N_4$ when $k=3$.
  Also, $\Pi(\cdot,\cdot)$ is defined in \eqref{def_pi}.
  Our goal is to show that there exists $\te>0$ such that
  \EQQS{
    \sup_{t\in [0,T]}|I_{t,k}|
    \le (CK)^{k-1}T^\te \|w\|_{Z_{T}^{s-1}}\|w\|_{L_T^\I H_x^{s-1}}.
  }
  Now we consider the following contribution on $I_{t,k}$:
  \begin{itemize}
    \item $N_4\gtrsim N_1/k$ (when $k\ge 3$),
    \item $N_1\gg kN_4$ and $N_2\gtrsim N_3$ (or $N_2\gtrsim N_3$ when $k=2$),
    \item $N_1\gg kN_4$ and $N_2\ll N_3$ (or $N_2\ll N_3$ when $k=2$).
  \end{itemize}

  \noindent
  \textbf{Case 1:} $N_4\gtrsim N_1/k$.
  Notice that $kN_3\ge kN_4\gtrsim N_1\ge N_2$.
  We have
  \EQQS{
    |I_{t,k}|
    &\les K^{k-3}
      \sum_{\substack{N_2\le N_1\les kN_4\le kN_3 ,\\N\gg 1}}
      N^{2(s-1)}
      \|\Pi (P_{N_1}w,P_{N_2}w)\|_{L_{T,x}^2}
      \|P_{N_3}zP_{N_4}z\|_{L_{T,x}^2}.
  }
  It suffices to consider the contribution of $\p_x P_N^2 P_{N_1}w P_{N_2}w$ in $\Pi (P_{N_1}w,P_{N_2}w)$, as $N_1\ge N_2$.
  For simplicity, we only treat the case $s< 1$.
  The case $s\ge 1$ can be handled easily, since there is more room to place derivatives to the functions.
  When $N_2\ge 1$, Propositions \ref{prop_bistri} and \ref{prop_bistri_d1}, and the Young inequality show that for sufficiently small $T>0$ satisfying $G(K)T^{\frac{1}{2}}\ll 1$
  \EQQS{
    &\sum_{1\le N_2\le N_1\les kN_4\le kN_3}
     N_1^{2(s-1)}
     \|\p_x P_N^2 P_{N_1}w P_{N_2}w\|_{L_{T,x}^2}
     \|P_{N_3}z P_{N_4}z \|_{L_{T,x}^2}\\
    &\les T^{\frac{1}{2}}K^{k-2}G(K)^2\|w\|_{L_T^\I H_x^{s-1}}
     \sum_{\substack{1\le N_2\le N_1,\\N_1\les kN_4\le kN_3}}
     N_1^{2s-1+\frac{\ga+1-\al}{4}} N_2^{1-s}
     N_3^{\frac{\ga+1-\al}{4}} N_4^{-s}
     W_{N_1} V_{N_3}\\
    &\les k^2T^{\frac{1}{2}}K^{k-2}G(K)^2\|w\|_{L_T^\I H_x^{s-1}}
     \sum_{N_1\les kN_3}N_1^{s_0-1}N_3^{s_0}
     \bigg(\frac{N_1}{kN_3}\bigg)^{\frac{1}{2}}
     W_{N_1} V_{N_3}\\
    &\les k^2T^{\frac{1}{2}}K^{k-1}G(K)^2\|w\|_{L_T^\I H_x^{s-1}}^2,
  }
  where, with a slight abuse of notation, $W_{N_1}$ and $V_{N_3}$ are defined by
  \EQQS{
    W_{N_1}
    &:=T^{-\frac{1}{2}}\|P_{N_1}w\|_{L_{T,x}^2}
    +N_1^{1-s-\e}\|w\|_{L_T^\I H_x^{s-1}},\\
    V_{N_3}
    &:=T^{-\frac{1}{2}}\|P_{N_3}z\|_{L_{T,x}^2}
    +N_3^{-s-\e}\|z\|_{L_T^\I H_x^s}.
  }
  Here, we recall that $s_0$ satisfies $s\ge s_0\ge s(\al)\vee(1/2+\e)$, where $\e>0$ is sufficiently small (for example, $\e=s_0/4-1/8$).
  On the other hand, when $N_2=0$, we directly estimate $\|\p_x P_N^2 P_{N_1}w P_{0}w\|_{L_{T,x}^2}$ by applying the H\"older inequality and the Bernstein inequality.

  \noindent
  \textbf{Case 2:} $N_1\gg kN_4$ and $N_2\gtrsim N_3$ (or $N_2\gtrsim N_3$ when $k=2$).
  It suffices to consider the case $k\ge 3$, since for $k=2$, the nonlinear interactions are always nonresonant, which implies that we can apply the Bourgain-type estimates.
  See also \cite{MV15}.
  Notice that we have $N_1\sim N_2\gts N_3\ge N_4$.
  When $N_3\gg kN_4$ and $N_3\gg 1$, we apply Bourgain-type estimates as in \cite{MT22}.
  On the other hand, when $N_3\gg kN_4$ but $N_3\les 1$, we can estimate directly, following \eqref{eq4.2}.
  So, we may assume $N_3\les kN_4$.
  When $N_2\sim N_3$, then $N_1\sim N_2\sim N_3\ge N_4$ and $N_3\les kN_4$ corresponds to a configuration already treated in the Case 1.
  Thus, we may further assume that $N_2\gg N_3$.
  As in the Case 3 of the proof of Proposition \ref{prop_apri1}, by the Taylor theorem, we decompose
  \EQQS{
    &\int_\R\Pi(P_{N_1}w,P_{N_2}w)Fdx\\
    &=-\int_\R P_{N_1}w(\check{a}_1*P_{N_2}w)\p_x Fdx
    +\frac{1}{N}\int_\R \La_{a_2}(P_{N_1}w,P_{N_2}w)\p_x^2 Fdx,
  }
  where $a_1$ and $a_2$ are defined in \eqref{def_a1}, and $F=\prod_{j=3}^{k+1}P_{N_{j}}z$.
  We then apply the refined bilinear Strichartz estimates; that is, we use Proposition \ref{prop_bistri2} for the first term and Proposition \ref{prop_bistri_d1} for the second term, assuming sufficiently small $T>0$.
  To be precise, we apply bilinear estimates in the case $N_3\gg k$ and $N_4\ge 1$, whereas the case $N_3\les k$ is treated directly.
  Recall that the case $N_3\gg k$ and $N_4=0$ cannot happen since $N_3\les kN_4$.
  See also Case 2 in the proof of Proposition 6 in \cite{MT3}.

  \noindent
  \textbf{Case 3:} $N_1\gg kN_4$ and $N_2\ll N_3$ (or $N_2\ll N_3$ when $k=2$).
  As in Case 2, we consider only the case $k\ge 3$.
  To close the estimate, we consider two subcases.
  One case is when $kN_4\ll N_2$, and the other is when $kN_4\gts N_2$.
  In the first subcase, we can apply Bourgain's type estimates when $N_2\gg 1$, whereas \eqref{eq4.2} directly yields the desired bound when $N_2\les 1$.
  These estimates are valid for $s>1/2$, regardless of the value of $\al$.
  For details, see Subcase 3.1 in the proof of Proposition 5.1 in \cite{MT22}.
  In the second subcase, we use the refined bilinear estimate (Proposition \ref{prop_bistri2}), which is also valid for $s>1/2$, independently of the value of $\al$; see Case 3 in the proof of Proposition 6 in \cite{MT3}.
  This completes the proof.
\end{proof}

\subsection{Proof of Theorem \ref{theo1}}
\label{subsec_proof}

\begin{proof}
  The proof follows from  standard arguments of the energy method with Propositions \ref{prop_apri1} and \ref{prop_diff1} in hands.
  For details, see Section 6 of \cite{MT22}.
\end{proof}

\section*{Acknowlegdements}

The second author was supported by JSPS KAKENHI Grant Numbers JP23K19019, JP25K17287.
A part of this work was conducted during a visit of the second author at Institut Denis Poisson (IDP) of Universit\'e de Tours in France. The second author is deeply grateful to IDP for its kind hospitality.

\end{document}